\documentclass{amsart}
\usepackage{amsmath,amssymb,amsthm,mathtools}
\usepackage{enumerate}
\usepackage{xcolor}
\usepackage[colorlinks,citecolor=blue,urlcolor=blue,linkcolor=blue]{hyperref}
\usepackage{microtype}
\usepackage{etoolbox}

\usepackage{tikz-cd}
\usepackage{tikz}
\usepackage{marginnote}
\definecolor{mygray}{gray}{0.85}
\usepackage[linecolor=black, backgroundcolor=mygray,colorinlistoftodos,prependcaption,textsize=small]{todonotes}
\usepackage{xargs}  
\usepackage{xcolor}

\newcommand{\mrm}[1]{\mathrm{#1}}

\renewcommand{\leq}{\leqslant}
\renewcommand{\geq}{\geqslant}

\newcommand{\Aut}{\operatorname{Aut}}
\newcommand{\Inn}{\operatorname{Inn}}

\newcommand{\ASpe}{\operatorname{AutSpe}}
\newcommand{\ISpe}{\operatorname{InnSpe}}
\newcommand{\FAut}{\operatorname{FAut}}
\newcommand{\FInn}{\operatorname{FInn}}
\newcommand{\FAutfin}{\operatorname{FAut}_{\mathrm{fin}}}
\newcommand{\FInnfin}{\operatorname{FInn}_{\mathrm{fin}}}
\newcommand{\PartConj}{\operatorname{PartConj}}
\newcommand{\PartConjsc}{\operatorname{PartConj}_{\mathrm{sc}}}
\newcommand{\Spe}{\operatorname{Spe}}
\newcommand{\ad}{\operatorname{ad}}
\newcommand{\supp}{\operatorname{supp}}
\newcommand{\first}{\operatorname{first}}
\newcommand{\last}{\operatorname{last}}
\newcommand{\link}{\operatorname{link}}
\newcommand{\Pc}{\operatorname{Pc}}
\newcommand{\esupp}{\operatorname{esupp}}

\newcommand{\PStab}{\operatorname{PStab}}
\newcommand{\N}{\mathbb N}
\newcommand{\Z}{\mathbb Z}
\newcommand{\cG}{\mathcal G}

\newcommand{\cl}[1]{\overline{#1}}

\AtBeginEnvironment{proof}{%
  \setlength{\parindent}{0pt}%
  \setlength{\parskip}{\smallskipamount}%
}

\numberwithin{equation}{section}

\theoremstyle{plain}
\newtheorem{theorem}{Theorem}[section]
\newtheorem{proposition}[theorem]{Proposition}
\newtheorem{lemma}[theorem]{Lemma}
\newtheorem{corollary}[theorem]{Corollary}
\newtheorem{fact}[theorem]{Fact}

\theoremstyle{definition}
\newtheorem{definition}[theorem]{Definition}
\newtheorem{observation}[theorem]{Observation}
\newtheorem{notation}[theorem]{Notation}
\newtheorem{remark}[theorem]{Remark}

\title[Acylindrical hyperbolicity and pure conjugating automorphisms]{Acylindrical hyperbolicity and pure conjugating automorphisms in graph products of groups}

\author[G. Paolini]{Gianluca Paolini}
\address{Department of Mathematics ``Giuseppe Peano'', University of Torino, Via Carlo Alberto 10, 10123 Torino, Italy}
\email{gianluca.paolini@unito.it}

\author[J. Rabideau]{Jean-Luc Rabideau}
\address{Department of Mathematics, Vanderbilt University, 1326 Stevenson Center Ln, Nashville, TN 37240, USA}
\email{jean-luc.p.rabideau@vanderbilt.edu}

\thanks{Research of Gianluca Paolini was supported by project PRIN 2022 ``Models, sets and classifications'', prot. 2022TECZJA, and by INdAM Project 2024 (Consolidator grant) ``Groups, Crystals and Classifications''.}
\thanks{Research of Jean-Luc Rabideau was supported by the National Science Foundation Graduate Research Fellowship Program under Grant No. 2444112.  Any opinions, findings, and conclusions or recommendations expressed in this material are those of the author(s) and do not necessarily reflect the views of the National Science Foundation.}

\date{September 1, 2026}

\begin{document}

\begin{abstract}
We study graph products of groups over defining graphs of arbitrary
cardinality from two closely related viewpoints.  First, we characterize
acylindrical hyperbolicity.  If the defining graph is irreducible, has at
least two vertices, and has a finite star base, then every parabolically
full subgroup is either virtually cyclic or acylindrically hyperbolic.  For
vertex-full subgroups the finite star base condition is also necessary.  In
particular, the graph product itself is acylindrically hyperbolic exactly
when the graph has a finite star base and the group is not virtually cyclic,
or equivalently is not the infinite dihedral group.  We also give the
corresponding classification for reducible defining graphs.  Second,
motivated by our earlier work on
countable right-angled Coxeter groups, we study pure conjugating
automorphisms in the topology of pointwise convergence.  We prove that they
are topologically generated by finite-support factorwise automorphisms and
partial conjugations, and establish
closedness, density, exact-equality, and discreteness results.  The same
finite star base condition that appears in the acylindrical hyperbolicity
criterion governs closedness and discreteness in the star-connected case;
for arbitrary graphs, its natural refinement is the existence of a finite
component witness set.
Finally, for graph products of abelian groups we
obtain a canonical topological semidirect decomposition of the pure
\mbox{conjugating automorphism group.}
\end{abstract}

\maketitle
\setcounter{tocdepth}{1}
\tableofcontents

\section{Introduction}

Graph products of groups were introduced by Green in \cite{Green1990}.  Let
$\Gamma$ be a simplicial graph and let
$\cG=(G_v)_{v\in V(\Gamma)}$ be a family of non-trivial groups.  The graph
product $G=\Gamma\cG$ is obtained from the free product of the vertex groups
by imposing commutation between elements of $G_u$ and $G_v$ whenever $u$ and $v$ are
adjacent.  Graph products interpolate between free and direct products and
include, as the two most familiar uniform examples, right-angled Artin
groups and right-angled Coxeter groups.

\smallskip
This paper grew directly out of our previous work
\cite{PaoliniRabideau2026}.  There, in order to realize non-Archimedean
Polish groups as outer automorphism groups, we studied the pointwise
convergence topology on automorphism groups of countable right-angled
Coxeter groups.  Tits' decomposition of $\mrm{Aut}(W)$, for $W$ a right-angled Coxeter group, reduced an important part of that
problem to understanding when the special automorphism group coincides with
the inner automorphism group.  Our approach was to separate that equality
into two topological questions, density and closedness, and to translate both of these questions into graph-theoretic properties.  The density part
was governed by so-called {\em star-connectedness}, while the closedness part led us to
introduce what we called {\em the finite star base condition}.  The present project asks how far
that mechanism extends from right-angled Coxeter groups to arbitrary graph
products, without countability assumptions.  This leads naturally to two
questions, one geometric and one topological, namely:
\begin{enumerate}[(1)]
\item when is $G=\Gamma\cG$ acylindrically hyperbolic (for $\Gamma$ of arbitrary cardinality)?
\item what is the pointwise-topological structure of the automorphisms that
preserve the conjugacy class of each specified vertex group?
\end{enumerate}

At first these questions appear unrelated.  Their connection is the main
organizing principle of the paper.  For $v \in V(\Gamma)$, write
\[
 N^*(v)=\{v\}\cup\{u\in V(\Gamma):uR_\Gamma v\},
 \qquad
 Z_\Gamma=\{v\in V(\Gamma):N^*(v)=V(\Gamma)\}.
\]
A subset $A\subseteq V(\Gamma)$ is a \emph{star base} if
$\bigcap_{a\in A}N^*(a)=Z_\Gamma$, and it is a \emph{finite star base} when
$A$ is finite\footnote{Notice that here $A$ is allowed to be empty: with the convention that
the empty intersection is $V(\Gamma)$, the empty set is a star base exactly
when every vertex is universal, or equivalently when $\Gamma$ is complete;
compare \cite{PaoliniRabideau2026}.}.  Thus a finite star base gives finitely
many witnesses to the
failure of a non-universal vertex to commute with the whole graph product.
We call $\Gamma$ \emph{irreducible} if its complement graph is connected,
and \emph{star-connected} if $\Gamma\setminus N^*(v)$ has at most one
connected component for every $v \in V(\Gamma)$.  Here a \emph{star} means
a set of the form $N^*(v)$.  Finally, a \emph{component witness set}
is a set meeting every connected component of every graph
$\Gamma\setminus N^*(v)$.

These notions clarify both the connection and the division of labour among
the results.  In the irreducible case, a finite star base is exactly the
finite combinatorial datum needed to obtain acylindrical hyperbolicity.  On
the automorphism side, partial conjugations can act independently on the
components of $\Gamma\setminus N^*(v)$.  Consequently, in order to force a
pure conjugating automorphism which preserves finitely many vertex groups to
preserve every vertex group, the chosen finite set of vertices must meet
each such component.  This is precisely the stronger requirement that the
graph admit a finite component witness set.  Every component witness set is
a star base, and on
a star-connected graph the converse holds.  Consequently, in the
star-connected case the same finite star base condition governs both the
geometric classification and the closedness and discreteness phenomena for
pure conjugating automorphisms.
Star-connectedness itself prevents different components
of a deleted star from carrying independent partial conjugations, and hence
turns density of partial conjugations and factorwise automorphisms into density of inner and factorwise
automorphisms.
We now state the main results of our paper in the two directions separately and then
return to their common finite conditions.

\subsection{Acylindrical hyperbolicity}
The class of acylindrically hyperbolic groups, defined by Osin in \cite{Osin2016}, is a wide class of groups exhibiting features of negative curvature. This class is broad enough to include many classes of groups admitting natural actions on hyperbolic spaces, including all non-elementary
hyperbolic and relatively hyperbolic groups, while retaining many strong structural properties.

Recall that a subgroup of a graph product $G$ is called \textit{parabolic}
if it is a conjugate of a subgroup generated by the groups
$\{G_v\}_{v \in T}$ for some $T \subseteq V(\Gamma)$.  A subgroup is called
\textit{parabolically full} if it is not contained in a proper parabolic
subgroup of $G$.
Minasyan--Osin \cite[Theorem~2.12]{MinasyanOsin2015} proved that, for a
finite irreducible defining graph with at least two vertices, every parabolically full subgroup is either virtually
cyclic or acylindrically hyperbolic.  Every finite graph has a finite star
base, and our first main result gives the same conclusion for arbitrary
defining graphs after replacing finiteness by having a finite star base.

\begin{theorem}\label{thm:intro-ah}
Let $G=\Gamma\cG$ be a graph product of non-trivial groups over an
irreducible graph $\Gamma$ with at least two vertices, and suppose $\Gamma$ has a finite star base. Suppose that $H \leq G$ is parabolically full. Then $H$ is either virtually cyclic or acylindrically hyperbolic.
\end{theorem}

This is Theorem~\ref{thm:pf_ah}. The proof builds on the theory of parabolic subgroups developed in \cite{AntolinMinasyan2015}, and uses the action of $G$ on the Bass-Serre tree of a standard amalgam as in \cite{MinasyanOsin2015}. More precisely, Proposition \ref{prop:parabolic-fullness-finite-supports} gives a useful alternative characterization of parabolically full subgroups when $\Gamma$ is an infinite graph. This characterization is used to construct an element $h \in H$ whose support contains a walk in $\Gamma^c$ visiting all vertices of a finite star base. This element is a loxodromic WPD element for the action of $G$ on the Bass-Serre tree.
Here WPD stands for \emph{weak proper discontinuity}; see
\cite[Definition~2.5]{Osin2016}.

The finite star base hypothesis is sufficient uniformly for all
parabolically full subgroups, but it need not be necessary for an individual
one: Remark~\ref{rem:counterexample} gives a parabolically full,
acylindrically hyperbolic subgroup over an irreducible graph with no finite
star base.  Necessity is recovered for the class of
\emph{vertex-full} subgroups, namely those meeting every vertex group
non-trivially.  Since the full graph product is vertex-full, the next result
also gives the exact criterion for an irreducible graph product itself to be
acylindrically hyperbolic.  Thus Theorems~\ref{thm:intro-ah} and
\ref{thm:intro-subgroup} complement each other: the former applies to all
parabolically full subgroups under a finite star base hypothesis, while the
latter characterizes that hypothesis for vertex-full subgroups.

\begin{theorem}\label{thm:intro-subgroup}
Let $G=\Gamma\cG$, where $\Gamma$ is irreducible and has at least two
vertices, and let $H\leq G$ satisfy $H\cap G_v\neq1$ for every
$v\in V(\Gamma)$.  Then $H$ is virtually cyclic or acylindrically
hyperbolic if and only if $\Gamma$ has a finite star base.  Moreover, the
virtually cyclic case can occur only when $\Gamma$ has two vertices and
$H\cap G_v\cong C_2$ for both of them.
\end{theorem}

 The forward direction of the proof uses the fact that an acylindrically hyperbolic
group has a finite subset whose centralizer is its center. For an irreducible graph
product the center is trivial, and we show that the existence of such a finite subset
is equivalent to the existence of a finite star base.  The converse uses the observation that any such subgroup $H$ is necessarily parabolically full (Proposition \ref{prop:vertex_full_are_pf}).

Decomposing an arbitrary graph product over the
connected components of the complement graph also gives the following
reducible form, recorded as Corollary~\ref{cor:arbitrary-ah}.

\begin{corollary}\label{cor:intro-arbitrary-ah}
Let $(\Gamma_i)_{i\in I}$ be the irreducible components of $\Gamma$.  Then
$G=\Gamma\cG$ is acylindrically hyperbolic if and only if $I$ is finite,
exactly one factor $G_{V(\Gamma_i)}$ is acylindrically hyperbolic, and every
other component consists of a single vertex carrying a finite group.
\end{corollary}

For context, the finite star base condition had previously appeared, under
the name \emph{coneless set}, in Ferov's work on pointwise-inner
endomorphisms \cite{Ferov2016}.  Valiunas constructed explicit
acylindrical actions on quasi-trees under a bounded-degree hypothesis
\cite{Valiunas2021}; every bounded-degree graph has a finite star base,
but the converse is false.  Oyakawa studies the geometry of graph products over
infinite defining graphs \cite{Oyakawa2026}.  For finite defining graphs,
Cohen classifies the stronger property of admitting a non-elementary
acylindrical action on a simplicial tree \cite{Cohen2026} and Genevois classifies relative hyperbolicity \cite[Theorem 8.35]{Genevois2017}.

\subsection{Pure conjugating automorphisms}
Automorphism groups of graph products have been studied from several
complementary perspectives.  The classical generating results for
right-angled Artin groups are due to Servatius and Laurence
\cite{Servatius1989,Laurence1995}, while Tits described the automorphism
groups of right-angled Coxeter groups \cite{Tits1988}.  For graph products
of abelian groups, structural and generating results were obtained by
Guti\'errez--Piggott--Ruane, Corredor--Guti\'errez, and
Charney--Ruane--Stambaugh--Vijayan
\cite{GutierrezPiggottRuane2012,CorredorGutierrez2012,CharneyRuaneStambaughVijayan2010}.

The main algebraic input in this paper is a theorem of Genevois and Martin: for a
finite defining graph, conjugating automorphisms are generated by local
automorphisms and partial conjugations \cite{GenevoisMartin2019}.
Genevois and Varghese study algebraic properties of the pure conjugating
automorphism group, denoted $\operatorname{PConjAut}(G)$ in their work
\cite{GenevoisVarghese2020}.  Genevois' later memoir concerns geometric and
algebraic properties of full automorphism groups under hypotheses on the
vertex groups \cite{Genevois2024}.  More recently, Escalier and Horbez gave
generators for full automorphism groups of finite graph products under
natural maximality assumptions on the chosen graph-product decomposition
\cite[Theorem~1.5]{EscalierHorbez2026}.

We equip $\Aut(G)$ with the topology of pointwise convergence,
where $G$ is discrete.  This is especially natural in the present setting:
the defining graph may have arbitrary cardinality
 while a basic neighborhood records agreement on
only finitely many elements.  Consequently, finite induced subgraphs and
canonical retractions can be used to solve finite pieces of a problem, and
topological closure records exactly the passage from all finite pieces to the
whole graph product.

We use four natural subgroups of $\Aut(G)$.  The group $\ASpe(G)$ consists
of the automorphisms sending each $G_v$ to a conjugate of the same vertex
group; this is the same as the pure conjugating automorphism group.  Its subgroup
$\ISpe(G)$ consists of those automorphisms whose restriction to each $G_v$
is conjugation by an element of the ambient graph product, where the
conjugator may depend on $v$.  The groups $\FAut(G)$ and $\FInn(G)$ preserve
every vertex group setwise and act, respectively, by arbitrary and inner
automorphisms on the individual factors.

\begin{center}
\small
\renewcommand{\arraystretch}{1.2}
\begin{tabular}{c|p{0.27\textwidth}|p{0.42\textwidth}}
 & Image of $G_v$ & Action after transporting back to $G_v$ \\
\hline
$\ASpe(G)$ & a conjugate of $G_v$ & an arbitrary automorphism of $G_v$ \\
$\ISpe(G)$ & a conjugate of $G_v$ & an inner automorphism of $G_v$ \\
$\FAut(G)$ & $G_v$ & an arbitrary automorphism of $G_v$ \\
$\FInn(G)$ & $G_v$ & an inner automorphism of $G_v$
\end{tabular}
\end{center}

Our first result is a topological generation theorem. It identifies dense subgroups that will later be shown to be closed, and hence equal to the ambient special automorphism groups, under the finite star base or finite component witness set hypotheses.
We
write
$ \PartConjsc(G)$
for the subgroup generated by single-component partial conjugations.  The
finite-support subgroups of $\FAut(G)$ and $\FInn(G)$ are denoted,
respectively, by
$\FAutfin(G)$ and $\FInnfin(G).$

\begin{theorem}\label{thm:intro-generation}
For every graph product $G=\Gamma\cG$,
\[
 \ASpe(G)=\cl{\FAutfin(G)\PartConjsc(G)}
\]
and
\[
 \ISpe(G)=
 \cl{\FInnfin(G)\PartConjsc(G)}^{\,\ISpe(G)}.
\]
Moreover, $\ASpe(G)$ is closed in $\Aut(G)$, and $\ISpe(G)$ is closed in
$\Aut(G)$ if and only if $\Inn(G_v)$ is closed in $\Aut(G_v)$ for every
$v\in V(\Gamma)$.
\end{theorem}

The generation statement is Theorem~\ref{thm:partial-density}, and
the closedness statement is Theorem~\ref{thm:special-closedness}.  The proof
uses the finite-graph theorem of Genevois--Martin only after retracting onto
a suitable finite induced subgraph.  Partial conjugations of that finite
graph are then lifted back to the ambient graph product so as to agree on a
prescribed finite set.  This is precisely the kind of finite approximation
for which the pointwise topology is natural.

Star-connectedness turns these topological generators into inner and factorwise
automorphisms, and a finite star base turns density into equality.

\begin{theorem}\label{thm:intro-special}
For every graph product $G=\Gamma\cG$, the following are equivalent:
\begin{enumerate}[(1)]
\item $\Gamma$ is star-connected;
\item $\ASpe(G)=\cl{\Inn(G)\FAut(G)}$;
\item $\ISpe(G)=\cl{\Inn(G)\FInn(G)}^{\,\ISpe(G)}$.
\end{enumerate}
If $\Gamma$ also has a finite star base, then the closures may be omitted in
(2) and (3).
\end{theorem}

This is Theorem~\ref{thm:exact-special}.  It displays the two ingredients
already present in our earlier right-angled Coxeter work in their general
form: star-connectedness gives density, while the finite star base gives
closedness.  When $\Gamma$ is not star-connected, a non-entire partial
conjugation cannot even be approximated by inner times factorwise
automorphisms.

For graphs which are not star-connected, the correct finite condition must
remember every component on which a partial conjugation may act.  This is
the role of finite component witness sets.  They imply closedness of
$\PartConj(G)\FAut(G)$ and, relatively, of
$\PartConj(G)\FInn(G)$ (see Proposition \ref{prop:product_closed}). Combining with Theorem \ref{thm:intro-generation} therefore yields the following.

\begin{corollary}
    Let $G = \Gamma\cG$. If $\Gamma$ has a finite component witness set, then \[
 \ASpe(G)=\FAut(G)\PartConjsc(G)
\]
and
\[
 \ISpe(G)=
 \FInn(G)\PartConjsc(G).
\]
\end{corollary}
This is Corollary \ref{cor:product_equality}.
Finite component witness sets also yield the following exact discreteness
criteria.

\begin{theorem}\label{thm:intro-component-witness}
The group $\PartConj(G)\FInn(G)$ is discrete if and only if $\Gamma$ has a
finite component witness set, every $\Inn(G_v)$ is discrete in
$\Aut(G_v)$, and $\Inn(G_v)=1$ for all but finitely many vertices.  Likewise,
$\PartConj(G)\FAut(G)$ is discrete if and only if $\Gamma$ has a finite
component witness set, every $\Aut(G_v)$ is discrete, and
$\Aut(G_v)=1$ for all but finitely many vertices.
\end{theorem}

These are Theorems~\ref{thm:pinn-discrete-general} and
\ref{thm:paut-discrete-general}.  On a star-connected graph, finite
component witness sets are exactly finite star bases, so this general
criterion reduces to the finite star base results of Section~\ref{sec:finite-star-base}.
This is the precise point at which the geometric and topological halves of
the paper meet.  In \cite{PaoliniRabideau2026}, the finite star base
condition controlled the closedness of the inner automorphism group of a
countable right-angled Coxeter group.  Here the same condition is the exact
criterion for acylindrical hyperbolicity in the irreducible case.  On the
automorphism side, it again controls closedness whenever star-connectedness
eliminates the possibility of independent partial conjugations on different
components of a deleted star.  For arbitrary graphs, finite component
witness sets record precisely this additional componentwise information.
\subsection{Abelian vertex groups and right-angled Coxeter groups}
When all vertex groups are abelian, the factorwise inner group is trivial
and the pure conjugating automorphism group admits a canonical topological semidirect-product decomposition; this is the content of the next theorem.

\begin{theorem}\label{thm:intro-abelian}
Let $G=\Gamma\cG$ be a graph product of non-trivial abelian groups.  There is
a canonical topological semidirect-product decomposition
\[
 \ASpe(G)=\ISpe(G)\rtimes\FAut(G),
 \qquad
 \FAut(G)\cong\prod_{v\in V(\Gamma)}\Aut(G_v).
\]
Moreover,
\[
 \ISpe(G)=\cl{\PartConjsc(G)}^{\,\ISpe(G)}.
\]
\end{theorem}

The semidirect decomposition is Theorem~\ref{thm:abelian-splitting}; the
topological-generation statement follows from
Theorem~\ref{thm:partial-density}.  Thus both the kernel and the
quotient are explicit, even at infinite rank.

Right-angled Coxeter groups occur when every vertex group is $C_2$.  In
that case $\FAut(W)=\FInn(W)=1$ and
$\ASpe(W)=\ISpe(W)=\Spe(W)$, so the preceding results recover the
density--closedness mechanism for star-connected graphs of \cite{PaoliniRabideau2026} and extend the results to arbitrary graphs.  In particular, $\Spe(W)$ is topologically
generated by single-component partial conjugations, and star-connectedness
is equivalent to $\Spe(W)=\cl{\Inn(W)}$.  We record these consequences only
as applications in the final subsection.

\subsection{Organization of the paper}
Section~\ref{sec:preliminaries} collects facts on graph products, parabolic subgroups, the
pointwise convergence topology, and acylindrical hyperbolicity used later.
Section~\ref{sec:ah} proves that parabolically full subgroups are virtually cyclic or acylindrically hyperbolic and classifies acylindrically hyperbolic graph products in both the irreducible and reducible cases. Section~\ref{sec:special} introduces the four
special automorphism groups, proves their basic closedness properties, and
establishes topological generation by finite-support factorwise
automorphisms and single-component partial conjugations.
Section~\ref{sec:finite-star-base} treats star-connected graphs and finite
star bases, obtaining density, exact equality, and 
discreteness criteria.  Section~\ref{sec:component-witness} replaces finite
star bases by finite component witness sets and proves the corresponding
results for arbitrary graphs.  Finally, Section~\ref{sec:special-abelian}
proves the canonical semidirect decomposition for abelian vertex groups and
briefly records the consequences for right-angled Coxeter groups.
\section{Preliminaries}\label{sec:preliminaries}

\subsection{Graphs and finite star bases}

Throughout, a graph $\Gamma=(V(\Gamma),R_\Gamma)$ is simplicial: the relation
$R_\Gamma$ is symmetric and irreflexive.  For $v\in V(\Gamma)$, put
\[
 N_\Gamma(v)=\{w\in V(\Gamma):vR_\Gamma w\},
 \qquad
 N_\Gamma^*(v)=N_\Gamma(v)\cup\{v\}.
\]
We omit the subscript $\Gamma$ in $R_\Gamma, N_\Gamma$ etc., when the graph is clear.  If $A\subseteq V(\Gamma)$,
then $\Gamma_A$ denotes the induced subgraph on $A$.  The complement graph is
denoted by $\Gamma^{\mathrm c}$. The \textit{link} of a subset $S \subset V(\Gamma)$ is \begin{align*}
 \link_\Gamma(S) &=\{z\in V(\Gamma)\setminus S:
                   z\text{ is adjacent to every }s\in S\} \\
                &= \bigcap_{s \in S} N_{\Gamma}(s). 
                   \end{align*}

\begin{definition}\label{def:universal-star-base}
The set of \emph{universal vertices} of $\Gamma$ is
$Z_\Gamma=\{v\in V(\Gamma):N^*(v)=V(\Gamma)\}$.  A subset
$A\subseteq V(\Gamma)$ is a \emph{star base} if
$\bigcap_{a\in A}N^*(a)=Z_\Gamma$.
It is a \emph{finite star base} if, in addition, $A$ is finite.  The empty
intersection is understood to be $V(\Gamma)$.
\end{definition}

\begin{definition}\label{def:irreducible-star-connected}
The graph $\Gamma$ is \emph{irreducible} if $\Gamma^{\mathrm c}$ is connected.
$\Gamma$ is \emph{star-connected} if, for every $v\in V(\Gamma)$, the graph
$\Gamma\setminus N^*(v)$ has at most one connected component.
\end{definition}

\begin{observation}\label{obs_irr} Notice that the empty graph is allowed in the definition of star-connectedness.  If
$\Gamma$ is irreducible and has at least two vertices, then
$Z_\Gamma=\varnothing$ and every vertex has a non-neighbor.
\end{observation} 

\begin{definition}\label{def:component-witness}
A subset $S\subseteq V(\Gamma)$ is a \emph{component witness set} if, for
every $u\in V(\Gamma)$ and every connected component $C$ of
$\Gamma\setminus N^*(u)$, one has $C\cap S\neq\varnothing$.
\end{definition}

The notions of component witness set and star base are related precisely by
star-connectedness.

\begin{lemma}\label{lem:witness-star-base}
Let $\Gamma$ be an arbitrary graph.  Every component witness set is a star
base.  Conversely, if $\Gamma$ is star-connected, then every star base is a
component witness set.
\end{lemma}

\begin{proof}
Let $S\subseteq V(\Gamma)$ be a component witness set, and let
$u\notin Z_\Gamma$.  Choose a connected component $C$ of
$\Gamma\setminus N^*(u)$ and an element $s\in C\cap S$.  Then
$u\notin N^*(s)$, so $u\notin\bigcap_{s\in S}N^*(s)$.  Hence $S$ is a star
base.

Conversely, suppose that $\Gamma$ is star-connected and let $A$ be a star
base.  If $u$ is universal, then $\Gamma\setminus N^*(u)=\varnothing$.
Otherwise, this graph has exactly one connected component, and the star-base
condition gives $a\in A\setminus N^*(u)$.  Thus that component meets $A$,
and $A$ is a component witness set.
\end{proof}

\subsection{Graph products and normal forms}

\begin{definition}\label{def:graph-product}
Let $\cG=(G_v)_{v\in V(\Gamma)}$ be a family of non-trivial groups.  The
\emph{graph product} $G=\Gamma\cG$ is defined to be
\[
 G=\left(\mathop{*}_{v\in V(\Gamma)}G_v\right)
   \bigg/
   \left\langle\!\left\langle
      [g,h]:g\in G_v,\ h\in G_w,\ vR_\Gamma w
   \right\rangle\!\right\rangle.
\]
The groups $G_v$ are the \emph{vertex groups}.  For
$A\subseteq V(\Gamma)$, the subgroup generated by the $G_v$, $v\in A$, is
denoted by $G_A$ and it is called a \emph{standard parabolic subgroup}.  A conjugate
of a standard parabolic subgroup is a \emph{parabolic subgroup}.
\end{definition}

A non-trivial element of a vertex group is a \emph{syllable}.  A syllable word
$s_1\cdots s_n$, where $s_i\in G_{v_i}\setminus\{1\}$, is \emph{reduced} if
its length cannot be shortened by the following operations: deleting a
trivial syllable, combining two consecutive syllables in the same vertex
group, and interchanging consecutive syllables whose vertex groups commute (also called syllable shuffling).
The normal-form theorem states that every element has a reduced
representative and that any two reduced representatives differ by a sequence
of interchanges of adjacent commuting syllables; see
\cite[Theorem~2.2]{AntolinMinasyan2015}, where the result is stated for graphs
of arbitrary cardinality.

\begin{notation}\label{not:support}
For $g\in G$, let $\supp(g)$ be the set of vertices whose groups occur in a
reduced word for $g$.  Let $\first(g)$, respectively $\last(g)$, be the set of
vertices that can occur as the first, respectively last, syllable of a
reduced representative of $g$.  For $F\subseteq G$, put
$\supp(F)=\bigcup_{f\in F}\supp(f)$.
\end{notation}

\begin{observation}\label{eq:support-product}
The normal-form theorem immediately gives
\[
 \supp(gh)\subseteq\supp(g)\cup\supp(h)\qquad(g,h\in G).
\]
In particular, reduction can delete vertex labels but cannot create a label
absent from both factors.
\end{observation}

\begin{notation}\label{retraction}
For every $A\subseteq V(\Gamma)$, the presentation gives a canonical
retraction $\rho_A:G\longrightarrow G_A$ which is the identity on $G_A$
and kills every $G_v$ with $v\notin A$.
\end{notation}

\begin{lemma}\label{lem:reduced-conjugate}
Fix $x\in V(\Gamma)$ and $g\in G$.  Suppose
$\last(g)\cap N^*(x)=\varnothing$.
Then, for every $1\neq h\in G_x$, the concatenation of a reduced word for
$g$, the syllable $h$, and the inverse word of $g$ is reduced.  Consequently,
$\supp(ghg^{-1})=\supp(g)\cup\{x\}$.
\end{lemma}

\begin{proof}
Let $s_1\cdots s_n$ be a reduced word for $g$.  The words
$s_1\cdots s_n$ and $s_n^{-1}\cdots s_1^{-1}$ are reduced, so any shortening
of $s_1\cdots s_n\,h\,s_n^{-1}\cdots s_1^{-1}$ would have to involve
syllables lying on opposite sides of the central
syllable $h$.

A syllable from the left copy of $g$ can combine with $h$ only if it belongs
to $G_x$ and can be shuffled to the right end of a reduced word for $g$.
This would put $x$ in $\last(g)$, contrary to the hypothesis.  Similarly,
no syllable from the right copy of $g^{-1}$ can combine with $h$.

Finally, suppose that a syllable from the left copy of $g$ could combine
with a syllable from the right copy.  In order to reach the central
syllable, it would first have to shuffle past all syllables to its right in
the left copy of $g$, and hence its vertex would belong to $\last(g)$.  It
would also have to commute with $h$, so this vertex would lie in $N(x)$.
This contradicts $\last(g)\cap N^*(x)=\varnothing$.  Thus the displayed
concatenation is reduced, and its support is $\supp(g)\cup\{x\}$.
\end{proof}

\subsection{Parabolic calculus}

We collect below facts on parabolic subgroups used in the sequel.  All of them hold
for graphs of arbitrary cardinality.

\begin{fact}\label{fact:parabolic-basics}
Let $A,B\subseteq V(\Gamma)$ and $v\in V(\Gamma)$.
\begin{enumerate}[(1)]
\item $G_A\cap G_B=G_{A\cap B}$; more generally, arbitrary intersections of
standard parabolics are obtained by intersecting their vertex sets.
\item The normalizer of $G_v$ is $N_G(G_v)=G_{N^*(v)}$.
\item The centralizer of $G_v$ is
$C_G(G_v)=Z(G_v)\times G_{N(v)}$, and, for $1\neq x\in G_v$,
$C_G(x)=C_{G_v}(x)\times G_{N(v)}$.
\item The center of $G$ is $Z(G)=\bigoplus_{v\in Z_\Gamma}Z(G_v)$.
\item If $K$ is parabolic and $gKg^{-1}\leq K$, then
$gKg^{-1}=K$.
\end{enumerate}
\end{fact}

\begin{proof}
Item~(1) is \cite[Lemma~3.3]{AntolinMinasyan2015}, item~(2) is the case
$K=G_v$ of \cite[Proposition~3.13]{AntolinMinasyan2015}, and item~(5) is
\cite[Lemma~3.9]{AntolinMinasyan2015}.

Let $1\neq x\in G_v$ and let $g\in C_G(x)$.  Then
$x\in gG_vg^{-1}\cap G_v$.  Apply
\cite[Proposition~3.4]{AntolinMinasyan2015} with both standard parabolics
equal to $G_v$.  The intersection is $hG_Ph^{-1}$ for some
$P\subseteq\{v\}$ and some $h\in G_v$.  Non-triviality forces
$P=\{v\}$, and therefore $gG_vg^{-1}\cap G_v=hG_vh^{-1}=G_v$.
Thus $G_v\leq gG_vg^{-1}$, or equivalently
$g^{-1}G_vg\leq G_v$.  Item~(5) gives equality, so
$g\in N_G(G_v)=G_{N^*(v)}=G_v\times G_{N(v)}$.  Writing $g=ab$ with
$a\in G_v$ and $b\in G_{N(v)}$, and observing that $b$ commutes with $x$, we
obtain $a\in C_{G_v}(x)$.  The reverse inclusion is immediate, and hence
$C_G(x)=C_{G_v}(x)\times G_{N(v)}$.
Intersecting this equality over $1\neq x\in G_v$ gives the formula for
$C_G(G_v)$.

Finally, if $z\in Z(G)$, then $z$ normalizes every vertex group.  By items
(1) and~(2), $z\in\bigcap_{v\in V(\Gamma)}G_{N^*(v)}
=G_{\bigcap_{v\in V(\Gamma)}N^*(v)}=G_{Z_\Gamma}$.
The standard parabolic $G_{Z_\Gamma}$ is the restricted direct product of the
universal vertex groups, and an element of it is central in $G$ precisely
when each of its non-trivial coordinates lies in the center of the
corresponding vertex group.  This proves item~(4).
\end{proof}

\begin{lemma}\label{lem:vertex-conjugates}
Let $v,w\in V(\Gamma)$ and $a,b\in G$.  If
$aG_va^{-1}\cap bG_wb^{-1}\neq 1$,
then $v=w$ and $aG_va^{-1}=bG_vb^{-1}$.
\end{lemma}

\begin{proof}
Conjugating by $b^{-1}$, we may assume that the intersection under
consideration is $gG_vg^{-1}\cap G_w$.  By
\cite[Proposition~3.4]{AntolinMinasyan2015}, there are
$P\subseteq\{v\}\cap\{w\}$ and $h\in G_w$ such that
$gG_vg^{-1}\cap G_w=hG_Ph^{-1}$.
Since the intersection is non-trivial, necessarily $v=w$ and $P=\{v\}$.
As $h\in G_v$, this gives $gG_vg^{-1}\cap G_v=hG_vh^{-1}=G_v$,
and hence $G_v\leq gG_vg^{-1}$.  Equivalently,
$g^{-1}G_vg\leq G_v$, so Fact~\ref{fact:parabolic-basics}(5), applied to
the parabolic subgroup $G_v$, yields $g^{-1}G_vg=G_v$.  Therefore
$gG_vg^{-1}=G_v$, which proves the claim.
\end{proof}

We will also use the following consequence of a lemma
of Genevois and Martin to simplify the proof of
Lemma~\ref{lem:support-propagation}.

\begin{lemma}\label{lem:parabolic-factorization}
Let $x,y\in V(\Gamma)$ and let $q\in G$.  If
$qG_yq^{-1}\leq G_{N^*(x)}$, then $q=hn$ for some
$h\in G_{N^*(x)}$ and $n\in G_{N^*(y)}$.
\end{lemma}

\begin{proof}
The displayed inclusion is equivalent to
$G_y\leq q^{-1}G_{N^*(x)}q$.  Applying
\cite[Lemma~3.17]{GenevoisMartin2019}\footnote{Notice that \cite[Lemma~3.17]{GenevoisMartin2019} is explicitly stated for an
arbitrary simplicial graph and its proof does not use the finiteness of the graph.} with
$\Lambda_1=\{y\}$, $\Lambda_2=N^*(x)$, and conjugator $q^{-1}$ gives
$q^{-1}\in G_yG_{N(y)}G_{N^*(x)}=G_{N^*(y)}G_{N^*(x)}$.
Taking inverses gives the required factorization.
\end{proof}

Another useful consequence of the lemma is the following.
\begin{lemma}\label{lem:alignment}
Suppose $S,T \subset V(\Gamma)$ and $p,q \in G$ are such that
$pG_Sp^{-1}\leq qG_Tq^{-1}$.
Then $S\subseteq T$.  Moreover, there is an element
$\ell\in G_{\link_\Gamma(S)}$ such that
$qG_Tq^{-1}=p\ell G_T\ell^{-1}p^{-1}$.
\end{lemma}

\begin{proof}
Set $r=q^{-1}p$.  The assumed inclusion is equivalent to
$G_S\leq r^{-1}G_Tr$.  Then \cite[Lemma~3.17]{GenevoisMartin2019} gives
$r^{-1}\in G_SG_{\link_\Gamma(S)}G_T$ and $S\subset T$.
Write $r^{-1}=s\ell t$ with $s\in G_S$, $\ell\in G_{\link(S)}$ and
$t\in G_T$. Then $q = pr^{-1} = ps\ell t$. Therefore, \begin{align*}
    qG_T q^{-1} &= ps\ell t G_T t^{-1} \ell^{-1}s^{-1} p^{-1} \\
&= ps\ell G_T \ell^{-1}s^{-1} p^{-1} && (t \in G_T) \\
&=p\ell sG_T s^{-1} \ell^{-1} p^{-1} && (\ell \in C_G(G_S),\; s\in G_S) \\
&= p\ell G_T \ell^{-1} p^{-1} && (s \in G_S \subset G_T).
\end{align*}
\end{proof}

\begin{fact}\label{eq:irreducible-decomposition}
The connected components of $\Gamma^{\mathrm c}$ are the
\emph{irreducible components} of $\Gamma$.  If $(\Gamma_i)_{i\in I}$ are the
corresponding induced subgraphs, then
\[
 G=\bigoplus_{i\in I}G_{V(\Gamma_i)}.
\]
Every element has non-trivial coordinates in only finitely many factors.
\end{fact}

For a finite subset of a graph product, there is a well-defined smallest parabolic subgroup containing it. This is the content of the following lemma.

\begin{lemma}[{\cite[Proposition~3.10]{AntolinMinasyan2015}}]\label{lem:finite-pc}
Let $F\subseteq G$ be finite. Then there is a unique minimal parabolic
subgroup containing $\langle F\rangle$. Moreover, if this subgroup has the
form $pG_Sp^{-1}$, then
$S\subseteq\supp(F)=\bigcup_{f\in F}\supp(f)$,
and in particular $S$ is finite.
\end{lemma}

\begin{definition}
Let $F\subseteq G$ be finite. The minimal parabolic subgroup from
Lemma~\ref{lem:finite-pc} is called the \emph{parabolic closure} of $F$
and is denoted by $\Pc(F)$.  If $\Pc(F)=pG_Sp^{-1}$,
then the set $S$ is called the \emph{essential support} of $F$ and is
denoted by $\esupp(F)=S$.
\end{definition}

\begin{lemma}\label{lem:monotonicity}
If $F\subseteq F'$ are finite subsets of $G$, then
$\Pc(F)\leq\Pc(F')$ and $\esupp(F)\subseteq\esupp(F')$.
\end{lemma}

\begin{proof}
The first assertion is minimality of $\Pc(F)$: any parabolic subgroup containing $F'$ also contains $F$.  The second follows from
Lemma~\ref{lem:alignment}.
\end{proof}

\begin{lemma}\label{lem:compatibility}
Let $U\subseteq V(\Gamma)$ and $X\subseteq G_U$.  If the parabolic closure
of $X$ exists, its computation inside the graph product $G_U$ agrees with
its computation in $G$.
\end{lemma}

\begin{proof}
This is \cite[Lemma~6.6]{MinasyanOsin2015}.  Its proof uses only the canonical retraction
$\rho_U:G\to G_U$ and the containment part of
Lemma~\ref{lem:alignment}, so the same proof works for an arbitrary defining
graph whenever the parabolic closure under discussion exists.
\end{proof}

\subsection{Partial conjugations}

\begin{definition}\label{def:partial-conjugation}
Let $u\in V(\Gamma)$, let $s\in G_u$, and let $C$ be a union of connected
components of $\Gamma\setminus N^*(u)$.  The \emph{partial conjugation}
$\pi(u,s,C)$ is defined on vertex groups by
\[
 \pi(u,s,C)(g)=
 \begin{cases}
   sgs^{-1},&g\in G_v\text{ and }v\in C,\\
   g,&g\in G_v\text{ and }v\notin C.
 \end{cases}
\]
The defining relations show that this extends to an automorphism.  We write
$\PartConj(G)$ for the subgroup generated by all partial conjugations. The partial conjugation is \emph{entire} if
$C=V(\Gamma)\setminus N^*(u)$, and \emph{non-entire} if
$\varnothing\subsetneq C\subsetneq V(\Gamma)\setminus N^*(u)$.
\end{definition}

An entire partial conjugation can be described as a product of a factorwise inner automorphism and an inner automorphism. Indeed, let $\alpha_{u,s^{-1}}\in\FInn(G)$ be the factorwise inner
automorphism which restricts to conjugation by $s^{-1}$ on $G_u$ and to the
identity on every $G_v$ with $v\neq u$.  Then
$\pi\bigl(u,s,V(\Gamma)\setminus N^*(u)\bigr)
=\ad(s)\circ\alpha_{u,s^{-1}}$.
The two automorphisms cancel on $G_u$; global conjugation by $s$ fixes the
neighboring vertex groups because they commute with $G_u$; and it conjugates
all vertex groups outside $N^*(u)$ by $s$.

\subsection{Pointwise-convergence topology}

For a discrete group $K$, give $K^K$ the product topology and $\Aut(K)$ the
subspace topology.  This is the \emph{topology of pointwise convergence}.  A
neighborhood basis at the identity consists of pointwise stabilizers of
finite subsets of $K$.  The group operations, including inversion, are
continuous.  When $K$ is countable, $\Aut(K)$ is a non-Archimedean Polish
group.

\begin{lemma}\label{lem:inn-discrete-centralizer}
For every group $K$, the following are equivalent.
\begin{enumerate}[(1)]
\item $\Inn(K)$ is discrete in $\Aut(K)$.
\item There is a finite set $F\subseteq K$ such that $C_K(F)=Z(K)$.
\end{enumerate}
If $K$ is countable, every countable closed subgroup of $\Aut(K)$ is
discrete.
\end{lemma}

\begin{proof}
The inner automorphism $\ad(k)$ fixes $F$ pointwise exactly when
$k\in C_K(F)$, and it is the identity exactly when $k\in Z(K)$.  Thus the
pointwise stabilizer of $F$ meets $\Inn(K)$ trivially precisely under (2).
For the last assertion, a countable closed subgroup of the Polish group
$\Aut(K)$ is itself a countable Polish group and hence is discrete by the
Baire category theorem.
\end{proof}

\subsection{Acylindrical hyperbolicity}

Let a group $K$ act by isometries on a metric space $X$.  The action is
\emph{acylindrical} if, for every $\varepsilon\geq0$, there exist
$R,N\geq0$ such that whenever $d(x,y)\geq R$, at most $N$ elements $k\in K$
satisfy $d(x,kx)\leq\varepsilon$ and $d(y,ky)\leq\varepsilon$.
An element $g$ is \emph{loxodromic} if an orbit map
$\Z\to X$, $n\mapsto g^nx$, is a quasi-isometric embedding.  A loxodromic
element $g$ is \emph{WPD} if, for every $\varepsilon\geq0$ and $x\in X$,
there is $M\geq1$ such that
\[
 \{k\in K:d(x,kx)\leq\varepsilon,
             \ d(g^Mx,kg^Mx)\leq\varepsilon\}
\]
is finite.  A group is \emph{acylindrically hyperbolic} if it admits a
non-elementary acylindrical action on a hyperbolic space.  We use the
equivalent WPD characterization from \cite[Theorem~1.2]{Osin2016}: a group
containing a loxodromic WPD element is either virtually cyclic or
acylindrically hyperbolic. 

For an action of a group $K$ on a space $X$, the pointwise
stabilizer of a set $Y \subseteq X$ is denoted by $\PStab_K(Y)$.

\begin{lemma}\label{lem:ah-finite-centralizer}
If $K$ is virtually cyclic or acylindrically hyperbolic, then there is a finite set $F\subseteq K$
such that $C_K(F)=Z(K)$.
Consequently, $\Inn(K)$ is discrete in $\Aut(K)$.
\end{lemma}

\begin{proof}
If $K$ is virtually cyclic, then it is finitely generated. Therefore, we may take $F$ to be a generating set of $K$.

Suppose instead $K$ is acylindrically hyperbolic. Choose two independent loxodromic elements $g,h$ for a non-elementary acylindrical action of $K$ on a hyperbolic space.  By \cite[Corollary~6.9]{Osin2016}, the centralizers $C_K(g)$ and $C_K(h)$ are virtually cyclic.  Their intersection $C=C_K(g)\cap C_K(h)$ is finite: otherwise it would contain an infinite-order loxodromic element with the same fixed points at infinity as both $g$ and $h$, contradicting independence.

For every $c\in C\setminus Z(K)$, choose $x_c\in K$ with
$cx_c\neq x_cc$, and put $F=\{g,h\}\cup\{x_c:c\in C\setminus Z(K)\}$.
An element centralizing $F$ lies in $C$ and cannot belong to $C\setminus Z(K)$ by the choice of $x_c$.  Thus $C_K(F)=Z(K)$.

The last assertion follows from Lemma~\ref{lem:inn-discrete-centralizer}.
\end{proof}

\section{Acylindrical hyperbolicity of graph products of groups}\label{sec:ah}

\begin{definition}
    We say a subgroup $H \leq G$ is \textit{parabolically full} if it is not contained in any proper parabolic subgroup.
\end{definition}

The first goal of this section is to prove Theorem~\ref{thm:pf_ah}, which states that parabolically full subgroups are either virtually cyclic or acylindrically hyperbolic. We first re-frame parabolic fullness in terms of essential supports of finite subsets.

\begin{proposition}
\label{prop:parabolic-fullness-finite-supports}
Let $G=\Gamma\mathcal G$ be a graph product of non-trivial groups, where
$V(\Gamma)\neq\varnothing$, and let $H\leq G$. Define
$$E_H\coloneqq\bigcup_{\substack{F\subseteq H\\ |F|<\infty}}\esupp(F).$$
Then $H$ is parabolically full if and only if $E_H = V(\Gamma)$.
\end{proposition}

\begin{proof}
Suppose first that $H$ is contained in a proper parabolic subgroup
$qG_Tq^{-1}$. For every finite subset $F\subseteq H$, minimality of
$\Pc(F)$ gives $\Pc(F)\leq qG_Tq^{-1}$.  The parabolic containment lemma
therefore implies $S_F=\esupp(F)\subseteq T$.
Hence $E_H\subseteq T$. Since $qG_Tq^{-1}$ is proper, we have
$T\subsetneq V(\Gamma)$, and therefore $E_H\neq V(\Gamma)$. Therefore, $E_H = V(\Gamma)$ implies $H$ is parabolically full.

Conversely, suppose that $E\coloneqq E_H \subsetneq V(\Gamma)$.
We show that $H$ is contained in a proper parabolic subgroup.

If $E=\varnothing$, then $\esupp(\{h\})=\varnothing$ for every $h\in H$.
Thus every element of $H$ is trivial, so $H=\{1\}$ and the conclusion is
immediate. We may therefore assume that $E_H\neq\varnothing$.

First suppose that there are non-adjacent vertices $s\in E$ and
$z\in V(\Gamma)\setminus E$. Choose a finite subset $F_0\subseteq H$ such
that $s\in S_0\coloneqq\esupp(F_0)$, and choose $p\in G$ such that
$\Pc(F_0)=pG_{S_0}p^{-1}$.
For an arbitrary $h\in H$, set $F_1=F_0\cup\{h\}$ and
$S_1=\esupp(F_1)$. Lemma \ref{lem:monotonicity} gives
$\Pc(F_0)\leq\Pc(F_1)$,
so Lemma  \ref{lem:alignment} gives an element
$\ell\in G_{\link_\Gamma(S_0)}$ such that
$\Pc(F_1)=p\ell G_{S_1}\ell^{-1}p^{-1}$.
Because $S_1\subseteq E$, we have $z\notin S_1$; hence $G_{S_1} \leq G_{V(\Gamma)\setminus \{z\}}$. Moreover,
$z\notin\link_\Gamma(S_0)$, since $s\in S_0$ and $z$ is not adjacent to
$s$; hence $\ell \in G_{V(\Gamma)\setminus \{z\}}$.  Consequently,
$h\in\Pc(F_1)\leq pG_{V(\Gamma)\setminus\{z\}}p^{-1}$.
Since $h\in H$ was arbitrary, and $p$ did not depend on the choice of $h$,
$H\leq pG_{V(\Gamma)\setminus\{z\}}p^{-1}$,
which is a proper parabolic subgroup.

It remains to consider the case in which every vertex of $E$ is adjacent
to every vertex of $V(\Gamma)\setminus E$. Then $G$ decomposes as the direct product
\[
    G=G_E\times G_{V(\Gamma)\setminus E}.
\]
Choose a finite subset $F_0\subseteq H$ with
$S_0\coloneqq\esupp(F_0)\neq\varnothing$, and choose $p\in G$ such that
$\Pc(F_0)=pG_{S_0}p^{-1}$.
For an arbitrary $h\in H$, the alignment lemma again gives
$\Pc(F_0\cup\{h\})=p\ell G_{S_1}\ell^{-1}p^{-1}$,
where $S_1\subseteq E$ and $\ell\in G_{\link_\Gamma(S_0)}$.

Using the direct-product decomposition of $G$, write
$\ell=\ell_E\ell_{E^c}$, where $\ell_E\in G_E$ and
$\ell_{E^c}\in G_{V(\Gamma)\setminus E}$.
Since $G_{V(\Gamma)\setminus E}$ commutes with $G_E$ and $S_1 \subset E$, the element
$\ell_{E^c}$ centralizes $G_{S_1}$. Therefore
$\ell G_{S_1}\ell^{-1}=\ell_EG_{S_1}\ell_E^{-1}\leq G_E$.
It follows that $h\in p\ell G_{S_1}\ell^{-1}p^{-1}\leq pG_Ep^{-1}$.
Since $h$ was arbitrary, $H\leq pG_Ep^{-1}$.
This parabolic subgroup is proper because $E\subsetneq V(\Gamma)$. Thus
$E_H\neq V(\Gamma)$ implies that $H$ is not parabolically full, completing
the proof of the equivalence.
\end{proof}

\begin{corollary}\label{cor:capturing}
    If $H$ is parabolically full, then for every finite subset
$D\subseteq V(\Gamma)$ there is a finite subset $F\subseteq H$ such that
$D\subseteq\esupp(F)$.
\end{corollary}
\begin{proof}
    Suppose that $H$ is parabolically full and let
$D\subseteq V(\Gamma)$ be finite. Since $E_H=V(\Gamma)$, for each $d\in D$
there is a finite subset $F_d\subseteq H$ such that $d\in\esupp(F_d)$.
Set $F=\bigcup_{d\in D}F_d$.
Then $F$ is finite, and monotonicity of finite parabolic closures gives
$\esupp(F_d)\subseteq\esupp(F)$ for every $d\in D$. Hence
$D\subseteq\esupp(F)$.
\end{proof}

We next isolate the special case of the parabolic-intersection formula used
throughout this section.

\begin{fact}\label{fact:parabolic-intersection}
Let $S\subseteq V(\Gamma)$ and $g\in G$.  If
$\first(g)\cap S=\last(g)\cap S=\varnothing$, then
$$gG_Sg^{-1}\cap G_S=G_{S\cap\bigcap_{v\in\supp(g)}N(v)}.$$
\end{fact}

\begin{proof}
This is the case $T=S$ of
\cite[Proposition~3.4]{AntolinMinasyan2015}.  In the notation of the proof of
that proposition, the hypotheses on the first and last syllables force the
initial $G_S$-prefix and the final $G_S$-suffix of $g$ to be trivial; that is, $g=g'$ and $h=1$.  The
resulting standard parabolic has vertex set
$Q \coloneqq S\cap\bigcap_{v\in\supp(g)}N(v)$.
\end{proof}

\begin{lemma}\label{lem:finite-core}
Let $\Gamma$ be irreducible and let $A$ be a finite star base.  There is a
finite set $B\supseteq A$ for which $\Gamma_B^{\mathrm c}$ is connected.
Moreover, if $S\supseteq B$, then
$\Gamma_S^{\mathrm c}$ is connected and $\link_\Gamma(S)=\varnothing$.
\end{lemma}

\begin{proof}
The claim is trivial when $|V(\Gamma)|=1$, so suppose $|V(\Gamma)| \geq 2$.

Since $\Gamma^c$ is connected, we may choose a finite walk in $\Gamma^c$ which visits every vertex of $A$. Let $B$ be the set of vertices on the walk. Then $B$ is finite,
contains $A$, and $\Gamma_B^{\mathrm c}$ is connected.

Now let $s\in S\setminus B$.  From the definition of a finite star base, there is $a\in A$ such that
$s\notin N^*(a)$; hence $s$ is joined to $a$ by an edge of
$\Gamma^{\mathrm c}$.  Thus every new vertex of $S$ attaches to the connected
subgraph $\Gamma_B^{\mathrm c}$, proving that $\Gamma_S^{\mathrm c}$ is
connected.  Finally,
\[
 \link_\Gamma(S)\subseteq\bigcap_{a\in A}N(a)
 \subseteq\bigcap_{a\in A}N^*(a)=\varnothing.
\]
\end{proof}

\begin{lemma}\label{lem:full-element}
Let $H\leq G$ be parabolically full, $\Gamma$ be irreducible, and let $B$ be as in
Lemma~\ref{lem:finite-core}.  There are a finite set
$F\subseteq H$, a finite set $S\supseteq B$, an element $p\in G$, and an
element $h\in p^{-1}\langle F\rangle p$ such that
$p^{-1}\Pc(F)p=G_S$ and $\Pc(h)=G_S$,
and $\Gamma_S^{\mathrm c}$ is connected and
$\link_\Gamma(S)=\varnothing$.
\end{lemma}

\begin{proof}
By Corollary \ref{cor:capturing}, we may choose a finite subset $F\subseteq H$ such that
$B\subseteq S:=\esupp(F)$.  Choose $p$ with
$p^{-1}\Pc(F)p=G_S$ and let $K = p^{-1}\langle F \rangle p \leq G_S$.

The graph $\Gamma_S$
is finite by Lemma \ref{lem:finite-pc} and irreducible by Lemma~\ref{lem:finite-core}, and the parabolic closure of $K$ in $G_S$ is $G_S$.  Therefore
\cite[Theorem~6.16]{MinasyanOsin2015} supplies a single element $h\in K$ whose parabolic closure in $G_S$
is $G_S$; note that we apply the theorem only in the finite graph product $G_S$.  By Lemma~\ref{lem:compatibility}, the parabolic closure of $h$ in $G$ is also  $G_S$; that is, $\Pc(h)=G_S$.
\end{proof}

\begin{lemma}\label{lem:wpd_existence}
    Suppose $\Gamma$ is irreducible, let $S$ be a finite and irreducible subset of vertices with $|S| \geq 2$ which contains a finite star base of $\Gamma$. Suppose $h \in G$ satisfies $\Pc(h) = G_S$, and let $v \in S$. Then $h$ is a loxodromic WPD element for the action of $G$ on the Bass-Serre tree $\mathcal{T}$ associated with the splitting \[G=G_A *_{G_C} G_B\] where $A = V(\Gamma) \setminus \{v\}$, $B = N^*(v)$, and $C = N(v)$.
\end{lemma}
\begin{proof}
    Let $d$ be a conjugate of $h$ of minimal graphical word length. Write
$d=U_1g_1U_2\cdots U_ng_nU_{n+1}$,
where each $1\neq g_i \in G_v$ and each $U_i$ is a word containing no $v$-syllables. After conjugating by $U_{n+1}$, we may assume that $U_{n+1} = 1$; this does not increase the graphical word length. Therefore, we write 
\begin{align}\label{eq:for_d}
    d = U_1 g_1 \cdots U_{n} g_n.
\end{align}
Since $d$ is a conjugate of $h$, the assumption $\esupp(h) = S$ and  the inclusion part of Lemma \ref{lem:alignment} imply $\supp(d) \supset S$. In particular, this implies that the word $d$ is not the length-one word $g_1$ or the word $U_1$, since $S$ contains both $v$ and other vertices.

If $n=1$, then $U_1 \not\in G_{N(v)}$: otherwise $\supp(d) \subseteq N^*(v)$, which contradicts $S \subseteq \supp(d)$, because $\Gamma_S^c$ is connected and $|S| \geq 2$ and therefore $v$ has a non-neighbor in $S$.

Now consider the case $n\geq 2$. If some $U_i$ with $2 \leq i \leq n$ represents an element of $G_{N(v)}$, we may combine $g_{i-1}$ and $g_{i}$ by commuting $g_i$ past $U_i$. Similarly, if $U_1$ represents an element of $G_{N(v)}$, conjugating by $g_n$ and combining $g_n$ with $g_1$ will yield a shorter conjugate of $h$.

Therefore, each $U_i$ represents an element of $G_A \setminus G_C$ and each $g_i$ represents an element of $G_B \setminus G_C$.
Therefore, the displayed equation \eqref{eq:for_d} for $d$ is an amalgamated product normal form which begins and ends with elements of different factors. It therefore acts loxodromically on the Bass-Serre tree $\mathcal{T}$: for all $k \in \N$, the amalgam word length of $d^k$ is $k$ times that of $d$.

Note also that $d$ is graphically cyclically reduced: if it were not, a conjugate of $d$ would have shorter graphical word length. It also has graphical length at least $2$, since $|S| \geq 2$.

We next show that $d$ is a WPD element for the action of $G$ on $\mathcal{T}$. It suffices to show that $w \coloneq d^{2|S| + 1}$ is a WPD element.

Fix a graphically cyclically reduced word $D$ representing $d$. Since $D$ is graphically cyclically reduced and of length at least $2$, the concatenation
\[
    W \coloneq D^{2|S| + 1}=\underbrace{D\cdots D}_{2|S|+1\text{ copies}}
\]
is a graphically reduced word representing $w$. 

    We reduce the word $W$ iteratively by the following procedure. Suppose that $\first(W) \cap A \neq \emptyset$. Then some reduced representative of $W$ begins with a syllable $y \in G_w$ for some $w \neq v$. Thus, we may write $W=yW'$ as a graphically reduced word. Replace $W$ by $W'$; this amounts to left-multiplying by $y^{-1} \in G_A$.

Similarly, if $\last(W) \cap A \neq \varnothing$, we may write $W = W'y$ for $y \in G_A$ and replace $W$ with $W'$.

Note that at each step, the resulting word $W'$ is graphically reduced, since it occurs as a subword of the graphically reduced word $W$, possibly after shuffling syllables.

Each reduction reduces the length of $W$, and so eventually, no such reductions are possible. Let $X$ denote the resulting reduced word and let $x \in G$ be the element it represents. By construction,
\[
x\in G_AwG_A
\qquad\text{and}\qquad
\first(x)\cap A=\last(x)\cap A=\varnothing.
\]

We claim that $S \subseteq \supp(x)$. By assumption, $v \in S \setminus A$, and $\Gamma_S^c$ is connected.  Number the copies of $D$ in $W = D^{2|S| + 1}$ as $D_1,D_2,\dots,D_{2|S|+1}$ and consider the copy $D_{|S| + 1}$.

Let $s\in S$, and choose $\alpha$ to be any syllable in $D_{|S|+1}$ with vertex label $s$. Choose a path $v=s_0,s_1,\dots,s_k=s$ in $\Gamma_S^c$ of length $k\leq |S|-1$. Consecutive vertices on this path are not adjacent in $\Gamma$, so syllables belonging to the corresponding vertex groups do not commute. For each $0 \leq i < k$, choose an $s_i$-syllable $\alpha_i$ in the copy $D_{i+1}$, and an $s_i$-syllable $\beta_i$ in the copy $D_{2|S|+1-i}$, and let $\alpha_k = \alpha = \beta_k$.

At every step of the reduction of $W$, we claim the syllables $\alpha_i,\beta_i$ occur in the order \[
\alpha_0,\alpha_1,\dots,\alpha_k=\alpha=\beta_k,\dots,\beta_1,\beta_0
\] in any graphical normal form of the current word. Indeed suppose these syllables occur, in that order, at a particular step in the reduction. Since the syllables occur in the order of a path in $\Gamma_S^{\mathrm c}$, their vertex groups do not commute, so none can be shuffled past a syllable which is adjacent in the list. Therefore, none of the syllables $\alpha_1,\dots,\alpha_k=\alpha=\beta_k,\dots,\beta_1$ can be shuffled past $\alpha_0$ or $\beta_0$. Since the syllables $\alpha_0,\beta_0$ are not in $G_A$, they cannot be removed from the next reduction step. It follows that none of the specified syllables can be removed in the next reduction step. This proves the claim.

In particular, the syllable $\alpha$ occurs in the final word $X$. Since $s \in S$ was arbitrary and $\alpha$ has vertex label $s$, we see $S \subseteq \supp(x)$.

Finally, we may apply Fact~\ref{fact:parabolic-intersection} to obtain 
\[
 xG_Ax^{-1}\cap G_A
 =G_P,
 \qquad
 P=\left(V(\Gamma) \setminus \{v\} \right)\cap\bigcap_{s\in\supp(x)}N(v).
\]
As $S\subseteq\supp(x)$, and $S$ contains a finite star base,
\[
 P\subseteq\bigcap_{s\in S}N(s)
   \subseteq\bigcap_{s\in S}N^*(s)=\varnothing.
\]
Thus $xG_Ax^{-1}\cap G_A=1$.  

Write $x = a_1 w a_2$ for $a_1,a_2 \in G_A$. Then \begin{align*}
    1 &= xG_Ax^{-1}\cap G_A \\
    &= a_1 w G_A w^{-1}a_1^{-1} \cap G_A \\
    &=a_1(wG_Aw^{-1} \cap G_A)a_1^{-1}.
\end{align*}

Therefore, \begin{align}\label{eq:trivial_intersection}
    wG_Aw^{-1} \cap G_A = 1.
\end{align}

From equation \eqref{eq:for_d}, we see that if $u \in \mathcal{T}$ is the standard vertex with stabilizer $G_A$, then the axis of $d$ is exactly $\bigcup_{k \in \mathbb Z}[d^ku,d^{k+1}u]$. In particular, both $u$ and $du$ lie on the axis of $d$.

The vertex $wu$ has stabilizer $wG_Aw^{-1}$. Therefore, \eqref{eq:trivial_intersection} gives $\PStab_G(u,wu) = 1$. Both of these vertices lie on the axis of $w$, which is the same as the axis of $d$. Therefore, \cite[Corollary 4.3]{MinasyanOsin2015} implies $w$ satisfies the WPD condition for the action of $G$ on $\mathcal{T}$. Therefore, so does its root $d$, and its conjugate $h$.
\end{proof}

\begin{theorem}\label{thm:pf_ah}
    Let $\Gamma$ be an irreducible graph on at least two vertices, and suppose $\Gamma$ has a finite star base.
    Let $H\leq G$ be a parabolically full subgroup. Then $H$ is virtually cyclic or acylindrically hyperbolic.
\end{theorem}
\begin{proof}
    Let $A \subset V(\Gamma)$ be a finite star base. Let $B \supset A$ be as in Lemma \ref{lem:finite-core}. Then Lemma \ref{lem:full-element} gives a set $S \supset B$, an element $p \in G$, and an element $h \in p^{-1}Hp$ such that $\Pc(h)=G_S$. Since $\Gamma$ is irreducible and has at least two vertices, any finite star base contains at least 2 vertices. Thus $|S| \geq 2$. Choose a vertex $v \in S$. Lemma \ref{lem:wpd_existence} therefore gives that $h$ is a loxodromic WPD element for the action of $G$ on a tree. Since $h$ is conjugate to the element $php^{-1} \in H$, the element $php^{-1}$ is a loxodromic WPD element for the restricted action of $H$. Thus, by \cite[Theorem~1.2]{Osin2016},  $H$ is either virtually cyclic or acylindrically hyperbolic.
\end{proof}

\begin{remark}\label{rem:counterexample}
    The existence of an acylindrically hyperbolic, parabolically full subgroup of an irreducible graph product does not imply that the underlying graph has a finite star base. For example, let $\Gamma^c$ be the infinite ray with vertex set $\N = \{0,1,2,\dots\}$. Then $\Gamma$ is irreducible and is easily seen not to have a finite star base. However, let
$G_0=F(x_1,x_2,\ldots)$
and, for $n\geq1$, let $G_n$ be any non-trivial group with a chosen
element $1\neq c_n\in G_n$. Set $h_n=x_nc_n$ and
$H=\langle h_n:n\geq1\rangle$.
The canonical retraction $G\to G_0$ sends $h_n$ to $x_n$. Hence the
elements $h_n$ freely generate $H$, since any relation among the $h_n$ would yield a relation under the retraction. So $H\cong F_\infty$.
In particular, $H$ is acylindrically hyperbolic.

Moreover, for every $n\geq1$, the element $h_n$ is graphically cyclically reduced
with $\esupp(h_n)=\{0,n\}$.  If $H\leq qG_Tq^{-1}$, then
$\Pc(h_n)\leq qG_Tq^{-1}$
for every $n$. Lemma \ref{lem:alignment} gives $\{0,n\}\subseteq T$ for
every $n$, and hence $T=V(\Gamma)$. Thus $H$ is not contained in any
proper parabolic subgroup.
\end{remark}

However, acylindrical hyperbolicity of the full graph product, or, more generally, of \textit{vertex-full} subgroups, does guarantee the existence of a finite star base.
\begin{definition}
Let $G=\Gamma\cG$ be a graph product of groups.  If $H\leq G$ satisfies
$H\cap G_v\neq1$ for every $v\in V(\Gamma)$, we say that $H$ is a
\emph{vertex-full subgroup} of $G$.
\end{definition}

The definition depends on the specified graph-product decomposition of
$G$.  Moreover, a vertex-full subgroup need not itself be the graph
product of the intersections $H\cap G_v$.  For example, let
$G=F(a,b)$ be the free group viewed as the graph product of the two vertex
groups $\langle a\rangle$ and $\langle b\rangle$.  Then
$H=\langle a^2,b^2,ab\rangle$ is a vertex-full subgroup of $G$, but
it is not the free product of $H\cap\langle a\rangle$ and
$H\cap\langle b\rangle$, since it is free of rank three.

\begin{lemma}\label{lem:vertex_trans_center}
    Let $G=\Gamma\cG$, where $\Gamma$ is irreducible and has at least two
vertices, and let $H\leq G$ be a vertex-full subgroup. Then $Z(H) = 1$.
\end{lemma}
\begin{proof}
    Let $z \in Z(H)$. For every vertex $v \in V(\Gamma)$, let $1\neq h_v \in G_v \cap H$. Then $zh_v=h_vz$, so $z\in C_H(h_v)\leq C_G(h_v)$.  By
Fact~\ref{fact:parabolic-basics}(3), $z\in G_{N^*(v)}$.  Since this holds
for every $v\in V(\Gamma)$, Fact~\ref{fact:parabolic-basics}(1) gives
$$z\in\bigcap_{v\in V(\Gamma)}G_{N^*(v)}
=G_{\bigcap_{v\in V(\Gamma)}N^*(v)}=G_\varnothing=1.$$
Therefore $Z(H)=1$.
\end{proof}

We now state the relationship between finite star bases and discreteness of
inner automorphisms.  The irreducible case is particularly clean because the
center is trivial.

\begin{proposition}\label{prop:irreducible-fsb}
Let $G=\Gamma\cG$, where $\Gamma$ is irreducible and has at least two
vertices. Let $H\leq G$ be a vertex-full subgroup.  The following are equivalent.
\begin{enumerate}[(1)]
\item $\Gamma$ has a finite star base.
\item There is a finite set $F\subseteq H$ such that $C_H(F)=1$.
\item $\Inn(H)$ is discrete in $\Aut(H)$.
\end{enumerate}
\end{proposition}

\begin{proof}
By Lemma~\ref{lem:vertex_trans_center}, one has $Z(H)=1$.  Thus the
equivalence of (2) and~(3) follows from
Lemma~\ref{lem:inn-discrete-centralizer}.

Assume (1), and let $A$ be a finite star base.  For every $a\in A$, choose
$1\neq x_a\in H \cap G_a$, and put $F=\{x_a:a\in A\}$.  By
Fact~\ref{fact:parabolic-basics}(3),
$C_H(x_a) \leq C_G(x_a)\leq G_{N^*(a)}$.  Consequently,
$$C_H(F)\leq \bigcap_{a\in A}G_{N^*(a)}
=G_{\bigcap_{a\in A}N^*(a)}=1,$$
where we used Fact~\ref{fact:parabolic-basics}(1) and the fact that
$A$ is a star base.  Thus (1) implies (2).

Conversely, assume (2), and put $A_0=\supp(F)$.  For every $a\in A_0$,
choose a vertex $a'\notin N^*(a)$; this is possible because $\Gamma$ has no
universal vertices.  Set $$A=A_0\cup\{a':a\in A_0\}.$$
We claim that $A$ is a star base.  Let $v\in V(\Gamma)$.  If $v\in A_0$,
then $v$ is not adjacent to the chosen vertex $v'\in A$.  Suppose that
$v\notin A_0$, and choose $1\neq x\in H \cap G_v$.  Since $C_H(F)=1$, there is
$f\in F$ with $[x,f]\neq1$.  If $v$ were adjacent to every vertex of
$\supp(f)$, then $x$ would commute with every syllable of a reduced word for
$f$, and hence with $f$, a contradiction.  Thus some
$a\in\supp(f)\subseteq A_0\subseteq A$ is not adjacent to $v$.  We have
proved that every vertex has a non-neighbor in $A$, equivalently
$\bigcap_{a\in A}N^*(a)=\varnothing$.  Hence (2) implies (1).
\end{proof}

\begin{proposition}\label{prop:vertex_full_are_pf}
        Vertex-full subgroups are parabolically full.
\end{proposition}

\begin{proof}
    Suppose $H \leq G$ is vertex-full and $H \leq q G_T q^{-1}$ for some $q \in G$ and $T\subset V(\Gamma)$. For every $v \in V(\Gamma)$, choose $1 \neq x_v \in H \cap G_v$. Then $G_v = \Pc(x_v) \leq qG_T q^{-1}$ by minimality of $\Pc(x_v)$. Lemma \ref{lem:alignment} then implies $v \in T$. Since this holds for every vertex, $V(\Gamma) = T$ and $H$ is parabolically full.
\end{proof}

\begin{theorem}\label{thm:ah-vertex_trans}
Let $G=\Gamma\cG$ be a graph product of non-trivial groups over an
irreducible graph $\Gamma$ with at least two vertices, and let $H\leq G$ be
a vertex-full subgroup.  Then $H$ is virtually cyclic or
acylindrically hyperbolic if and only if $\Gamma$ has a finite star base.
Furthermore, if $H$ is virtually cyclic, then $|V(\Gamma)|=2$ and
$H\cap G_v\cong C_2$ for both vertices $v\in V(\Gamma)$.
\end{theorem}

\begin{proof}
    First, suppose $H$ is acylindrically hyperbolic or virtually cyclic. By Lemma \ref{lem:ah-finite-centralizer}, there is a finite set $F \subset H$ such that $C_H(F) = Z(H)$. Thus, Proposition \ref{prop:irreducible-fsb} implies $\Gamma$ has a finite star base.

    Conversely, if $H$ is vertex-full, then $H$ is parabolically full by Proposition \ref{prop:vertex_full_are_pf}. Thus, if $\Gamma$ has a finite star base, then Theorem \ref{thm:pf_ah} implies $H$ is virtually cyclic or acylindrically hyperbolic.

It remains to discuss the virtually cyclic case. If
$|V(\Gamma)| \geq 3$, the connected graph $\Gamma^{\mathrm c}$ contains a
path $p,q,r$ on three distinct vertices.  Thus $G_q$ commutes with neither
$G_p$ nor $G_r$.  If $p$ and $r$ are adjacent in $\Gamma$, then $H$ contains a subgroup
\[
 \left((H \cap G_p)\times (H \cap G_r) \right)*(H\cap G_q);
\]
otherwise $H$ contains a subgroup
\[
 (H \cap G_p)*(H \cap G_q)*(H \cap G_r).
\]
In either case, $H$ contains a non-abelian free subgroup, so is not virtually cyclic.

Suppose instead $|V(\Gamma)| = 2$, and let $V(\Gamma) = \{v_1,v_2\}$. Then $H$ contains a subgroup isomorphic to $(H \cap G_{v_1}) * (H \cap G_{v_2})$, which contains non-abelian free groups unless $H \cap G_v \cong C_2$ for all $v \in V(\Gamma)$.
\end{proof}

By taking $H = G$, we reach a characterization of the graph products of groups which are acylindrically hyperbolic.

\begin{theorem}\label{thm:irreducible-ah}
Let $G=\Gamma\cG$ be a graph product of non-trivial groups over an
irreducible graph $\Gamma$ with at least two vertices.  Then $G$ is
acylindrically hyperbolic if and only if $\Gamma$ has a finite star base and
one of the following holds:
\begin{enumerate}[(a)]
\item $|V(\Gamma)|\geq3$; or
\item $|V(\Gamma)|=2$ and at least one vertex group is not isomorphic to
$C_2$.
\end{enumerate}
\end{theorem}
\begin{proof}
    If $G$ is acylindrically hyperbolic, then taking $H=G$ in Theorem~\ref{thm:ah-vertex_trans} implies that $\Gamma$ has a finite star base. Furthermore, if both conditions (a) and (b) fail, then $G \cong D_\infty$, contradicting that $G$ is acylindrically hyperbolic. 

    Conversely, suppose $\Gamma$ has a finite star base and one of conditions (a) or (b) holds. Since $\Gamma$ has a finite star base, we see $G$ is either acylindrically hyperbolic or virtually cyclic by Theorem~\ref{thm:ah-vertex_trans}. However, if $G$ is virtually cyclic, then Theorem~\ref{thm:ah-vertex_trans} implies $|V(\Gamma)| = 2$ and $G \cap G_v = G_v \cong C_2$ for both vertices $v \in V(\Gamma)$, contradicting either condition (a) or (b). Therefore, $G$ is acylindrically hyperbolic.
\end{proof}

We now pass from irreducible graph products to arbitrary ones.

\begin{corollary}\label{cor:arbitrary-ah}
Let $G=\Gamma\cG$, and let $(\Gamma_i)_{i\in I}$ be the irreducible
components of $\Gamma$.  The following are equivalent.
\begin{enumerate}[(1)]
\item $G$ is acylindrically hyperbolic.
\item The set $I$ is finite, there is a unique $i_0\in I$ for which
$G_{V(\Gamma_{i_0})}$ is infinite, this factor is acylindrically hyperbolic,
and every other factor is finite.
\item The set $I$ is finite, and there is a unique $i_0\in I$ such that
either $\Gamma_{i_0}$ consists of a single vertex carrying an
acylindrically hyperbolic group, or $\Gamma_{i_0}$ has at least two
vertices and satisfies the criterion of
Theorem~\ref{thm:irreducible-ah}; and every other component consists of a
single vertex carrying a finite group.
\end{enumerate}
\end{corollary}

\begin{proof}
Write $G_i=G_{V(\Gamma_i)}$.  By~\eqref{eq:irreducible-decomposition},
$G=\bigoplus_{i\in I}G_i$.

Assume that $G$ is acylindrically hyperbolic.  Such a group cannot split as a
direct product of two infinite groups
\cite[Corollary~7.2(b)]{Osin2016}.  If $I$ were infinite, partition $I$ into
two non-empty subsets so that each corresponding direct sum is infinite;
this is possible because every $G_i$ is non-trivial, and yields a
contradiction.  Hence $I$ is finite.  The same direct-product obstruction
shows that at most one factor is infinite.  At least one factor is infinite
because an acylindrically hyperbolic group is infinite.  Let it be $G_{i_0}$. Then $G_{i_0}$ is an infinite normal subgroup of $G$, hence is acylindrically hyperbolic by \cite[Corollary 1.5]{Osin2016}. This proves
(2).

If an irreducible component has at least two vertices, it contains two
nonadjacent vertices, and the product of non-trivial elements from the
corresponding vertex groups has infinite order by the normal-form theorem.
Thus a finite irreducible factor consists of a single vertex with a finite
vertex group.  Combining this observation with
Theorem~\ref{thm:irreducible-ah} proves the equivalence of (2) and (3).

Finally, if (2) holds, write $G=G_{i_0}\times F$ with $F$ finite.  Pulling back a
non-elementary acylindrical action of $G_{i_0}$ on a hyperbolic space through the projection
$G\to G_{i_0}$ produces a non-elementary acylindrical action of $G$ on a hyperbolic space, since
the kernel $F$ is finite.  Thus (2) implies~(1).
\end{proof}

\section{Special automorphisms of graph products of groups}\label{sec:special}

\subsection{Definitions and finite graph products}

\begin{definition}\label{def:special-groups}
Let $G=\Gamma\cG$.  We define the following subgroups of $\Aut(G)$.
\begin{enumerate}[(1)]
\item $\ASpe(G)=\bigl\{\sigma\in\Aut(G):
      \forall v\ \exists p_v\in G,\
      \sigma(G_v)=p_vG_vp_v^{-1}\bigr\}$;
\item $\ISpe(G)=\bigl\{\sigma\in\Aut(G):
      \forall v\ \exists p_v\in G,\
      \sigma|_{G_v}=\ad(p_v)|_{G_v}\bigr\}$;
\item $\FAut(G)=\bigl\{\sigma\in\Aut(G):
      \sigma(G_v)=G_v\text{ for every }v\bigr\}$;
\item $\FInn(G)=\bigl\{\sigma\in\FAut(G):
      \sigma|_{G_v}\in\Inn(G_v)\text{ for every }v\bigr\}$.
\end{enumerate}
Finally, $\FAutfin(G)$ and
$\FInnfin(G)$ denote the subgroups of $\FAut(G)$ and $\FInn(G)$,
respectively, whose automorphisms are the identity at all but finitely many
vertices.
\end{definition}

The four groups are summarized in Table~\ref{tab:special-groups}; in the first
two rows, the element $p_v$ may depend on $v$.
\begin{table}[ht]
\centering
\small
\renewcommand{\arraystretch}{1.25}
\begin{tabular}{c|p{0.28\textwidth}|p{0.42\textwidth}}
 & Image of $G_v$ & Action after transporting back to $G_v$ \\
\hline
$\ASpe(G)$ & $p_vG_vp_v^{-1}$ & an arbitrary automorphism of $G_v$ \\
$\ISpe(G)$ & $p_vG_vp_v^{-1}$ & an inner automorphism of $G_v$ \\
$\FAut(G)$ & $G_v$ & an arbitrary automorphism of $G_v$ \\
$\FInn(G)$ & $G_v$ & an inner automorphism of $G_v$
\end{tabular}
\vspace{0.75em}
\caption{The four automorphism groups from Definition~\ref{def:special-groups}.}
\label{tab:special-groups}
\end{table}

The letter F in $\FAut$ and $\FInn$ stands for \emph{factorwise}.
These are the cases in which no vertex group is moved: elements of
$\FAut(G)$ and $\FInn(G)$ preserve every $G_v$ setwise, and elements of
$\FInn(G)$ induce inner automorphisms on all individual factors.  Thus the following inclusions hold:
\begin{enumerate}[(i)]
\item $\FInn(G)\leq \ISpe(G)\leq \ASpe(G)$;
\item $\FInn(G)\leq \FAut(G)\leq \ASpe(G)$.
\end{enumerate}

The groups $\FAut(G)$ and $\ISpe(G)$ are generally incomparable.  A
factorwise automorphism whose restriction to some vertex group is outer lies
in $\FAut(G)$ but not in $\ISpe(G)$, whereas an inner automorphism lies in
$\ISpe(G)$ but need not preserve the vertex groups setwise.

\begin{observation}
Every family $(\alpha_v)_{v\in V(\Gamma)}$ with
$\alpha_v\in\Aut(G_v)$ extends uniquely to an element of $\FAut(G)$, because
it preserves all the defining commutation relations.  Consequently, we have
\begin{equation}\label{eq:paut-products}
 \FAut(G)\cong\prod_{v\in V(\Gamma)}\Aut(G_v),
 \qquad
 \FInn(G)\cong\prod_{v\in V(\Gamma)}\Inn(G_v)
\end{equation}
as abstract groups.  Moreover,
\begin{equation}\label{eq:pinn-intersection}
 \FInn(G)=\FAut(G)\cap\ISpe(G).
\end{equation}
Indeed, let $\sigma\in\FAut(G)\cap\ISpe(G)$ and fix $v\in V(\Gamma)$.  Since
$\sigma\in\ISpe(G)$, there exists $p_v\in G$ such that
$\sigma|_{G_v}=\ad(p_v)|_{G_v}$.  Since $\sigma\in\FAut(G)$, we also have
$\sigma(G_v)=G_v$, and therefore $p_v\in N_G(G_v)=G_{N^*(v)}
=G_v\times G_{N(v)}$.  Write $p_v=a_vb_v$, with $a_v\in G_v$ and
$b_v\in G_{N(v)}$.  As $b_v$ centralizes $G_v$, it follows that
$\sigma|_{G_v}=\ad(a_v)|_{G_v}\in\Inn(G_v)$.  This holds for every
$v\in V(\Gamma)$, so $\sigma\in\FInn(G)$.  The reverse inclusion is
immediate, and hence $\FInn(G)=\FAut(G)\cap\ISpe(G)$.
\end{observation} 

The next theorem follows from a result of Genevois--Martin \cite{GenevoisMartin2019}.

\begin{theorem}\label{thm:finite-decomposition}
If $\Gamma$ is finite, then, with $\PartConj(G)$ as in
Definition~\ref{def:partial-conjugation}, we have
\[
 \ASpe(G)=\FAut(G)\PartConj(G),
\text{ and }
 \ISpe(G)=\FInn(G)\PartConj(G).
\]
\end{theorem}

\begin{proof}
Every partial conjugation and every factorwise inner automorphism belongs to $\ISpe(G)$, while every factorwise automorphism belongs to $\ASpe(G)$, so one inclusion in each equality is immediate.

Conversely,
let $\sigma\in\ASpe(G)$.  It is a conjugating automorphism in the terminology
of \cite{GenevoisMartin2019}.  By
the proofs of \cite[Lemma 3.12, Theorem 3.11]{GenevoisMartin2019}, there is a product $\alpha$ of
partial conjugations and a graph isometry $s$ such that
$\sigma\alpha(G_v)=G_{s(v)}$ for every vertex $v \in V(\Gamma)$. Since $\sigma$ and $\alpha$ both belong to $\ASpe(G)$, the groups $G_v$
and $G_{s(v)}$ are conjugate.  Lemma~\ref{lem:vertex-conjugates} therefore
gives $s(v)=v$ for every $v$, thus
$\sigma\alpha\in\FAut(G)$.  This proves
$\ASpe(G)=\FAut(G)\PartConj(G)$.

Now let $\sigma\in\ISpe(G)$.  Partial conjugations belong to $\ISpe(G)$, so
$\sigma\alpha\in\ISpe(G)\cap\FAut(G)=\FInn(G)$ by
\eqref{eq:pinn-intersection}.  This gives the second equality.
\end{proof}

\subsection{Topology of special automorphism groups}
In this section, we record useful facts about the topology of the subgroups from Definition \ref{def:special-groups}. These will be used throughout the paper and are important for understanding our topological generation results.

\begin{lemma}\label{lem:factorwise-product-topology}
Under the identifications~\eqref{eq:paut-products}, the subspace topologies
on $\FAut(G)$ and $\FInn(G)$ inherited from $\Aut(G)$ are precisely the
product topologies.
\end{lemma}

\begin{proof}
Let $\Phi\colon\prod_{v\in V(\Gamma)}\Aut(G_v)\longrightarrow\FAut(G)$
be the coordinatewise extension map.  This is an algebraic isomorphism by
\eqref{eq:paut-products}.

For every $v$, the restriction map
$\FAut(G)\to\Aut(G_v)$ is continuous in the pointwise-convergence
topologies.  Hence the inverse of $\Phi$, whose coordinates are precisely
these restriction maps, is continuous for the product topology.

Conversely, fix $x\in G$ and choose a reduced expression
$x=s_1\cdots s_n$, with $s_i\in G_{v_i}$.  For
$(\alpha_v)_v\in\prod_v\Aut(G_v)$, one has
$\Phi((\alpha_v)_v)(x)=\alpha_{v_1}(s_1)\cdots\alpha_{v_n}(s_n)$.
Thus evaluation at $x$ after applying $\Phi$ depends continuously on only
finitely many coordinates.  Since the pointwise-convergence topology on
$\FAut(G)$ is generated by evaluations at elements of $G$, the map $\Phi$
is continuous.  Therefore $\Phi$ is a homeomorphism.  Restricting it to
$\prod_v\Inn(G_v)$ gives the assertion for $\FInn(G)$.
\end{proof}

\begin{proposition}\label{prop:paut-closed}
The subgroup $\FAut(G)$ is closed in $\Aut(G)$.
\end{proposition}

\begin{proof}
Recall that a \emph{net} in a topological space $X$ is a family
$(x_i)_{i\in I}$ indexed by a directed set $(I,\leq)$.  It converges to
$x\in X$ if, for every neighborhood $U$ of $x$, there is $i_0\in I$ such
that $x_i\in U$ whenever $i\geq i_0$.  Nets characterize closed subsets:
a subset of $X$ is closed if and only if it contains the limit of every
convergent net of its elements.  We use nets rather than sequences because,
for graph products of arbitrary cardinality, the topology on $\Aut(G)$ need
not be first countable.

Let $(\tau_i)_{i\in I}$ be a net in $\FAut(G)$ converging to some
$\tau\in\Aut(G)$.  Convergence in the pointwise-convergence topology means
that, for every $x\in G$, there is $i_x\in I$ such that
$\tau_i(x)=\tau(x)$ for every $i\geq i_x$.

Fix $v\in V(\Gamma)$ and $x\in G_v$.  Since every $\tau_i$ belongs to
$\FAut(G)$, it preserves $G_v$ setwise, and hence $\tau_i(x)\in G_v$ for
every $i$.  For all sufficiently large $i$ we have
$\tau_i(x)=\tau(x)$, so $\tau(x)\in G_v$.  As this holds for every
$x\in G_v$, we obtain $\tau(G_v)\leq G_v$.

Inversion is continuous in the topological group $\Aut(G)$, and therefore
the inverse net $(\tau_i^{-1})_{i\in I}$ converges to $\tau^{-1}$.  Moreover,
each $\tau_i^{-1}$ still belongs to $\FAut(G)$.  Applying the preceding
argument to this inverse net gives $\tau^{-1}(G_v)\leq G_v$.
Applying $\tau$ to this inclusion yields $G_v\leq\tau(G_v)$.  Consequently,
$\tau(G_v)=G_v$ for every $v\in V(\Gamma)$, and hence
$\tau\in\FAut(G)$.  Thus $\FAut(G)$ is closed in $\Aut(G)$.
\end{proof}

\begin{theorem}\label{thm:special-closedness}
Let $G=\Gamma\cG$.
\begin{enumerate}[(1)]
\item $\ASpe(G)$ is closed in $\Aut(G)$.
\item $\ISpe(G)$ is closed in $\ASpe(G)$, equivalently in $\Aut(G)$, if and only if $\Inn(G_v)$ is closed in $\Aut(G_v)$ for every $v\in V(\Gamma)$.
\end{enumerate}
\end{theorem}

\begin{proof}
For (1), let $(\sigma_i)_{i\in I}$ be a net in $\ASpe(G)$ converging to
$\sigma\in\Aut(G)$.  Fix $v\in V(\Gamma)$ and choose $1\neq x_v\in G_v$.
By pointwise convergence, $\sigma_i(x_v)=\sigma(x_v)$ for all sufficiently
large $i$.  Hence, for sufficiently large $i$ and $j$, the conjugates
$\sigma_i(G_v)$ and $\sigma_j(G_v)$ have a common non-trivial element.
Lemma~\ref{lem:vertex-conjugates} implies that they coincide.  Thus there is
a conjugate $K_v$ of $G_v$ such that $\sigma_i(G_v)=K_v$ eventually.

\smallskip\noindent
For every $x\in G_v$, eventual equality $\sigma_i(x)=\sigma(x)$ gives
$\sigma(x)\in K_v$, and hence $\sigma(G_v)\leq K_v$.  Since inversion is
continuous, $\sigma_i^{-1}\to\sigma^{-1}$.  If $y\in K_v$, then eventually
$y\in\sigma_i(G_v)$, so $\sigma_i^{-1}(y)\in G_v$; pointwise convergence of
the inverse net gives $\sigma^{-1}(y)\in G_v$.  Thus
$K_v\leq\sigma(G_v)$, and therefore $\sigma(G_v)=K_v$.  Since this holds for
every $v$, we conclude that $\sigma\in\ASpe(G)$.

\smallskip\noindent
For (2), note first that, by~(1), $\ASpe(G)$ is closed in $\Aut(G)$.
Hence $\ISpe(G)$ is closed in $\ASpe(G)$ if and only if it is closed in
$\Aut(G)$.  Suppose first that $\ISpe(G)$ is closed, and fix
$v\in V(\Gamma)$, and let $(\alpha_i)_{i\in I}$ be a net in
$\Inn(G_v)$ converging in $\Aut(G_v)$ to $\alpha$.  Extend each $\alpha_i$, and also $\alpha$, by the
identity on all other vertex groups.  The resulting factorwise automorphisms $\widehat{\alpha_i}$ belong to $\ISpe(G)$ and converge pointwise to $\widehat{\alpha}$.  Since $\ISpe(G)$ is closed in $\Aut(G)$, we have $\widehat{\alpha}\in\ISpe(G)$.
Since $\widehat{\alpha}\in\FAut(G)$, equality
\eqref{eq:pinn-intersection} gives $\widehat{\alpha}\in\FInn(G)$, and
therefore $\alpha\in\Inn(G_v)$.  Thus $\Inn(G_v)$ is closed in
$\Aut(G_v)$ for every $v$.

\smallskip\noindent
Conversely, assume that every $\Inn(G_v)$ is closed in $\Aut(G_v)$,
and let $(\sigma_i)_{i\in I}$ be a net in $\ISpe(G)$ converging to
$\sigma\in\Aut(G)$.  By (1),
$\sigma\in\ASpe(G)$.  Fix $v\in V(\Gamma)$.  As above, there are $p\in G$ and an index
$i_0$ such that, for every $i\geq i_0$,
$\sigma_i(G_v)=\sigma(G_v)=pG_vp^{-1}$.  For every such $i$, choose $q_i\in G$ with
$\sigma_i|_{G_v}=\ad(q_i)|_{G_v}$.  Since
$q_iG_vq_i^{-1}=pG_vp^{-1}$, the element $p^{-1}q_i$ normalizes $G_v$.

By~\eqref{eq:pinn-intersection}, the factorwise extension of
$\beta_i=\ad(p^{-1})\circ\sigma_i|_{G_v}
=\ad(p^{-1}q_i)|_{G_v}$
by the identity on the other vertex groups belongs to
$\FAut(G)\cap\ISpe(G)=\FInn(G)$.  Thus
$\beta_i\in\Inn(G_v)$.  The net $(\beta_i)$ converges pointwise to
$\beta=\ad(p^{-1})\circ\sigma|_{G_v}$.  By closedness,
$\beta\in\Inn(G_v)$, and consequently $\sigma|_{G_v}$ is conjugation by an
element of $G$.  Since $v$ was arbitrary, $\sigma\in\ISpe(G)$.
\end{proof}

\subsection{Topological generation}

Our first main result in this section shows that finite-support factorwise
automorphisms together with single-component partial conjugations
topologically generate $\ASpe(G)$ and $\ISpe(G)$.  We use the following
compatibility of canonical retractions with elements of these two groups.

\begin{lemma}\label{lem:special-retraction}
Let $A\subseteq V(\Gamma)$ and let $M_A=\ker(\rho_A)$.  Then
$M_A=\left\langle\!\left\langle G_v:v\notin A\right\rangle\!\right\rangle$.
If $\sigma\in\ASpe(G)$, then $\sigma(M_A)=M_A$, and $\sigma$ induces an
automorphism $\sigma_A\in\ASpe(G_A)$ given by
$\sigma_A(x)=\rho_A(\sigma(x))$ for $x\in G_A$.
If $\sigma\in\ISpe(G)$, then $\sigma_A\in\ISpe(G_A)$.

Finally, the assignment $\sigma\mapsto\sigma_A$ defines homomorphisms
\[
 \ASpe(G)\longrightarrow\ASpe(G_A),
 \qquad
 \ISpe(G)\longrightarrow\ISpe(G_A).
\]
\end{lemma}

\begin{proof}
The description of $M_A$ follows directly from the graph-product
presentation: quotienting by the normal closure of all vertex groups outside
$A$ leaves precisely the presentation of $G_A$.

For $v\notin A$, the group $G_v$ lies in the normal subgroup $M_A$.  If
$\sigma(G_v)=pG_vp^{-1}$, then $\sigma(G_v)\leq pM_Ap^{-1} = M_A$.  Hence
$\sigma(M_A)\leq M_A$.  Applying the same argument to $\sigma^{-1}$ gives
equality, so $\sigma$ descends to an automorphism of
$G/M_A\cong G_A$.  For $v\in A$, if $\sigma(G_v)=pG_vp^{-1}$, then
$\sigma_A(G_v)=\rho_A(p)G_v\rho_A(p)^{-1}$,
which proves that $\sigma_A\in\ASpe(G_A)$.  If $\sigma$ restricts to
conjugation by $p$ on $G_v$, then $\sigma_A$ restricts to conjugation by
$\rho_A(p)$.

Finally, if $\sigma, \tau \in \ASpe(G)$, then using the defining property of the homomorphisms $(\sigma \circ \tau)_A, \sigma_A,$ and $\tau_A$, we see
\begin{align*}
(\sigma \circ \tau)_A \circ \rho_A & = \rho_A \circ (\sigma \circ \tau) \\
&= \sigma_A \circ \rho_A \circ \tau \\
&= \sigma_A \circ \tau_A \circ \rho_A.
\end{align*}
Since $\rho_A$ is surjective, this implies
$(\sigma\circ\tau)_A=\sigma_A\circ\tau_A$, proving the last claim.
\end{proof}

Our other tool is the following graph-theoretic observation, which will allow us to lift partial conjugations of standard parabolic subgroups to partial conjugations of the full graph product.

\begin{lemma}\label{lem:finite-path-closure}
For every finite $S\subseteq V(\Gamma)$ there is a finite set
$T\supseteq S$ such that, for every $u\in S$ and every connected component
$D$ of $\Gamma\setminus N^*_{\Gamma}(u)$, the set $D\cap S$ is either empty
or is contained in a single connected component of
$\Gamma_T\setminus N^*_{\Gamma_T}(u)$.
\end{lemma}

\begin{proof}
For every subset $T\subseteq V(\Gamma)$ and every $u\in T$, one has
$N^*_{\Gamma_T}(u)=N^*_{\Gamma}(u)\cap T$, and therefore
\begin{equation}\label{eq:finite-ambient-deleted-star}
 \Gamma_T\setminus N^*_{\Gamma_T}(u)
 =\Gamma_{T\setminus N^*_{\Gamma}(u)}.
\end{equation}
In particular, the graph on the left is the induced subgraph of
$\Gamma\setminus N^*_{\Gamma}(u)$ on the vertices belonging to $T$.

There are only finitely many triples $(u,x,y)$ with $u\in S$ and
$x,y\in S\setminus N^*_{\Gamma}(u)$.  Whenever $x$ and $y$ lie in the same
connected component of $\Gamma\setminus N^*_{\Gamma}(u)$, choose a finite
path between them in that graph.  Let $T$ be the union of $S$ with the
vertex sets of all the chosen paths; then $T$ is finite.

Fix $u\in S$ and a connected component $D$ of
$\Gamma\setminus N^*_{\Gamma}(u)$.  If $x,y\in D\cap S$, then the path
chosen for $(u,x,y)$ has all its vertices in $T$ and avoids
$N^*_{\Gamma}(u)$.  By~\eqref{eq:finite-ambient-deleted-star}, it is a path
in $\Gamma_T\setminus N^*_{\Gamma_T}(u)$.  Thus every two vertices of $D\cap S$ lie in the same connected component of $\Gamma_T\setminus N^*_{\Gamma_T}(u)$, which proves the assertion.
\end{proof}

Let $\PartConjsc(G)$ be the subgroup generated by those partial
conjugations $\pi(u,a,C)$ for which $C$ is a single connected component of
$\Gamma\setminus N^*(u)$.  Then
\[
 \PartConjsc(G)\leq\PartConj(G)\leq
 \cl{\PartConjsc(G)}=\cl{\PartConj(G)}.
\]
Indeed, on a prescribed finite subset of $G$, a partial conjugation supported
on an arbitrary union of components agrees with the product of the finitely
many single-component partial conjugations meeting the supports of those
elements.  When $\Gamma$ is finite, all the displayed inclusions are
equalities.

\begin{theorem}\label{thm:partial-density}
One has
\[
 \ISpe(G)=\cl{\FInnfin(G) \PartConjsc(G)}^{\ISpe(G)}
\] and \[
\ASpe(G)=\cl{\FAutfin(G) \PartConjsc(G)}.
\]
If $\Gamma$ is finite, the closures may be omitted.
\end{theorem}

\begin{proof}
The reverse inclusions follow from Theorem \ref{thm:special-closedness}, the inclusions $\FInn(G) \leq \ISpe(G)$ and $\FAut(G) \leq \ASpe(G)$ noted earlier, and the observation that $\PartConjsc(G) \leq \ISpe(G)$.
We prove the forward inclusions by finite approximation.

\smallskip\noindent
Fix $\sigma\in\ASpe(G)$ and a finite set $F\subseteq G$, and put
$S=\supp(F)\cup\supp(\sigma(F))$.
Choose a finite set $T\supseteq S$ as in
Lemma~\ref{lem:finite-path-closure}.  By
Lemma~\ref{lem:special-retraction}, $\sigma$ induces automorphisms
$\sigma_T\in\ASpe(G_T)$ and $\sigma_S\in\ASpe(G_S)$, given by
\[
 \sigma_T(x)=\rho_T(\sigma(x))\quad(x\in G_T),
 \qquad
 \sigma_S(x)=\rho_S(\sigma(x))\quad(x\in G_S),
\]
where $\rho_A$ denotes the canonical retraction from
Notation~\ref{retraction}.

\smallskip\noindent
For $S\subseteq T$, write
$\rho^T_S=\rho_S|_{G_T}\colon G_T\longrightarrow G_S$
for the canonical retraction.  Applying
Lemma~\ref{lem:special-retraction} to the graph product $G_T$ and the subset
$S\subseteq T$, every $\tau\in\ASpe(G_T)$ induces an automorphism
$r_{T,S}(\tau)\in\ASpe(G_S)$ given by
$r_{T,S}(\tau)(x)=\rho^T_S(\tau(x))$ for $x\in G_S$.
The canonical retractions satisfy
$\rho^T_S\circ\rho_T=\rho_S$.  Hence, for every $x\in G_S$,
\[
 r_{T,S}(\sigma_T)(x)
 =\rho^T_S(\sigma_T(x))
 =\rho^T_S(\rho_T(\sigma(x)))
 =\rho_S(\sigma(x))
 =\sigma_S(x),
\]
and therefore
\begin{equation}\label{eq:restriction-functoriality}
 r_{T,S}(\sigma_T)=\sigma_S.
\end{equation}

\smallskip\noindent
Since $T$ is finite, Theorem~\ref{thm:finite-decomposition} gives 
$\ASpe(G_T)=\FAut(G_T)\PartConj(G_T)$. Moreover, a partial conjugation of $G_T$ whose
conjugated set is a union of components is a finite product of partial conjugations whose conjugated sets are single components.  We may therefore write
\begin{equation}\label{eq:sigmaT-partial-product}
 \sigma_T=\theta \pi_1\cdots\pi_m,
 \qquad
 \pi_i=\pi_T(u_i,a_i,C_i)
 \quad(1\leq i\leq m),
\end{equation}
where $\theta \in \FAut(G_T)$ is a factorwise automorphism of $G_T$, $u_i\in T$, $1\neq a_i\in G_{u_i}$, and $C_i$ is a connected
component of $\Gamma_T\setminus N^*_{\Gamma_T}(u_i)$.  Thus $\pi_i$
conjugates every $G_v$ with $v\in C_i$ by $a_i$ and fixes all the other
vertex groups of $G_T$.

\smallskip\noindent
For each $i\in\{1,\ldots,m\}$, we construct
$\widetilde\pi_i\in\PartConjsc(G)$ which preserves $G_S$ setwise and
satisfies
\begin{equation}\label{eq:lift-restricted-partial}
 \widetilde\pi_i|_{G_S}=r_{T,S}(\pi_i).
\end{equation}
Suppose first that $u_i\notin S$.  Since $a_i\in G_{u_i}$ and the canonical
retraction $\rho^T_S$ kills every vertex group indexed by $T\setminus S$, we
have $\rho^T_S(a_i)=1$.  Thus, for $v\in S$ and $g\in G_v$,
\[
 r_{T,S}(\pi_i)(g)=\rho^T_S(\pi_i(g))
 =\begin{cases}
   \rho^T_S(a_i)\,\rho^T_S(g)\,\rho^T_S(a_i)^{-1}=g,&v\in C_i,\\
   \rho^T_S(g)=g,&v\notin C_i.
  \end{cases}
\]
Hence $r_{T,S}(\pi_i)$ is the identity on $G_S$, and we set
$\widetilde\pi_i=1$.

\smallskip\noindent
Suppose now that $u_i\in S$.  If $C_i\cap S=\varnothing$, then $\pi_i$
fixes every vertex group of $G_S$, so again
$r_{T,S}(\pi_i)=1$ and we set $\widetilde\pi_i=1$.  Assume therefore that
$C_i\cap S\neq\varnothing$.

By~\eqref{eq:finite-ambient-deleted-star}, the graph
$\Gamma_T\setminus N^*_{\Gamma_T}(u_i)$ is an induced subgraph of
$\Gamma\setminus N^*_{\Gamma}(u_i)$.  Since $C_i$ is connected in the
former graph, it is contained in a unique connected component $D_i$ of the
latter graph.  
We claim that
\begin{equation}\label{eq:ambient-finite-component-intersection}
 D_i\cap S=C_i\cap S.
\end{equation}
The inclusion $C_i\cap S\subseteq D_i\cap S$ follows immediately from
$C_i\subseteq D_i$.  For the reverse inclusion, apply
Lemma~\ref{lem:finite-path-closure} to the ambient component $D_i$.  It gives
a connected component $E_i$ of
$\Gamma_T\setminus N^*_{\Gamma_T}(u_i)$ such that
$D_i\cap S\subseteq E_i$. Since $\varnothing \neq C_i \cap S \subset D_i \cap S \subset E_i$, we have $C_i \cap E_i \neq \varnothing$. Since both $C_i$ and $E_i$ are connected components of $\Gamma_T\setminus N^*_{\Gamma_T}(u_i)$, this yields $E_i = C_i$. Thus $D_i \cap S \subset E_i \cap S = C_i \cap S$, proving the claim.

Define $\widetilde\pi_i=\pi(u_i,a_i,D_i)$.
This is a single-component partial conjugation of the ambient graph product,
so $\widetilde\pi_i\in\PartConjsc(G)$.  For $v\in S$ and $g\in G_v$,
equality
\eqref{eq:ambient-finite-component-intersection} gives
$v\in D_i$ if and only if $v\in C_i$.  Moreover, by assumption $u_i\in S$, so
$a_i\in G_S$ and $\rho^T_S(a_i)=a_i$.  Consequently,
\[
 \widetilde\pi_i(g)
 =\begin{cases}
   a_iga_i^{-1},&v\in C_i,\\
   g,&v\notin C_i
  \end{cases}
 =\rho^T_S(\pi_i(g))
 =r_{T,S}(\pi_i)(g).
\]
The two automorphisms therefore agree on every vertex group generating
$G_S$, which proves~\eqref{eq:lift-restricted-partial}; in particular,
$\widetilde\pi_i$ preserves $G_S$ setwise.

\smallskip\noindent

Finally, choose a lift $\widetilde{\theta} \in \FAutfin(G)$ such that $\widetilde{\theta}|_{G_S} = r_{T,S}(\theta)$; for example, choose $\widetilde{\theta}$ to agree with $\theta$ on vertices of $T$ and be trivial elsewhere.

Set $\widetilde\sigma=\widetilde{\theta}\widetilde\pi_1\cdots\widetilde\pi_m$.
Observe $\widetilde\theta$ and each $ \widetilde\pi_i$ preserves the whole subgroup $G_S$, not merely the
finite set $F$.  Consequently, their restrictions to $G_S$ may be composed. As noted in Lemma~\ref{lem:special-retraction}, the assignment $r_{T,S}$ is a homomorphism. Using this fact, together with
\eqref{eq:lift-restricted-partial},
\eqref{eq:sigmaT-partial-product}, and
\eqref{eq:restriction-functoriality}, we obtain
\[
 \begin{aligned}
 \widetilde\sigma|_{G_S}
 &=\bigl(\widetilde{\theta}|_{G_S} \bigr)\bigl(\widetilde\pi_1|_{G_S}\bigr)\cdots
   \bigl(\widetilde\pi_m|_{G_S}\bigr)\\
 &=r_{T,S}(\theta)r_{T,S}(\pi_1)\cdots r_{T,S}(\pi_m)\\
 &=r_{T,S}(\theta \pi_1\cdots\pi_m)
 =r_{T,S}(\sigma_T)
 =\sigma_S.
 \end{aligned}
\]

\smallskip\noindent
For every $f\in F$, both $f$ and $\sigma(f)$ belong to $G_S$ by the
definition of $S$.  Therefore
\[
 \widetilde\sigma(f)=\sigma_S(f)=\rho_S(\sigma(f))=\sigma(f).
\]
We have found an element $\widetilde\sigma \in \FAutfin(G)\PartConjsc(G)$ agreeing with $\sigma$ on the
arbitrary finite set $F$.  Hence
$\sigma\in\cl{\FAutfin(G)\PartConjsc(G)}$, and so
$$\ASpe(G)=\cl{\FAutfin(G)\PartConjsc(G)}.$$

If instead we had started with $\sigma\in\ISpe(G)$, then
$\sigma_T\in\ISpe(G_T)$.  Hence $\theta$ and its lift
$\widetilde{\theta}$ could be chosen in $\FInn(G_T)$ and $\FInn(G)$,
respectively.  This yields an element
$\widetilde\sigma\in\FInnfin(G)\PartConjsc(G)$ agreeing with $\sigma$ on
the arbitrary finite set $F$.  Therefore
$\ISpe(G)=\cl{\FInnfin(G)\PartConjsc(G)}^{\,\ISpe(G)}$.

Finally, if $\Gamma$ is finite, we have $\FInn(G) = \FInnfin(G)$ and $\FAut(G) = \FAutfin(G)$. Furthermore, every partial conjugation may be written as a product of single-component partial conjugations. The result therefore follows from Theorem \ref{thm:finite-decomposition}.
\end{proof}

This result implies the following, which generalizes the density theorem
\cite[Theorem~5.4]{PaoliniRabideau2026} from right-angled Coxeter groups to
arbitrary graph products.  Indeed, in a RACG one has
$\ASpe(G)=\ISpe(G)$ and $\FInn(G)=\FAut(G)=1$.

\begin{theorem}\label{thm:special-density}
If $\Gamma$ is star-connected, then
\begin{equation}\label{eq:ASpe-density}
 \ASpe(G)=\cl{\Inn(G)\FAut(G)}
\end{equation}
and
\begin{equation}\label{eq:ISpe-density}
 \ISpe(G)=\cl{\Inn(G)\FInn(G)}^{\,\ISpe(G)}.
\end{equation}
\end{theorem}

\begin{proof}
The reverse inclusions in both equalities are immediate.  For the forward
inclusions, note that when $\Gamma$ is star-connected, every partial
conjugation is entire.  The formula following
Definition~\ref{def:partial-conjugation} therefore gives
$\PartConj(G)\leq\Inn(G)\FInn(G)$.  Combining this with
Theorem~\ref{thm:partial-density} yields
\begin{align*}
 \ASpe(G)
 &=\cl{\PartConj(G)\FAut(G)}\\
 &\leq\cl{\Inn(G)\FInn(G)\FAut(G)}
 =\cl{\Inn(G)\FAut(G)}.
\end{align*}
The same argument gives
$\ISpe(G)=\cl{\Inn(G)\FInn(G)}^{\,\ISpe(G)}$.
\end{proof}

\section{Finite star bases and star-connected graphs}\label{sec:finite-star-base}
Having found topological generating sets for $\ASpe(G)$ and $\ISpe(G)$,
we now seek graph-theoretic conditions under which the corresponding dense
subgroups are closed and hence equal to the ambient special automorphism
groups.  We first treat star-connected graphs.  In this case every partial
conjugation is entire, and the relevant finite-detection condition is the
existence of a finite star base.

\begin{lemma}\label{lem:universal-inner-factorwise}
If $z\in G_{Z_\Gamma}$, then $\ad(z)\in\FInn(G)$.
\end{lemma}

\begin{proof}
For every $v\in V(\Gamma)$, the definition of $Z_\Gamma$ and
Fact~\ref{fact:parabolic-basics}(2) give
$z\in G_{N^*(v)}=N_G(G_v)$.  Thus $\ad(z)\in\FAut(G)$, while it
plainly belongs to $\ISpe(G)$.  Equation~\eqref{eq:pinn-intersection}
therefore gives $\ad(z)\in\FInn(G)$.
\end{proof}

\begin{proposition}\label{prop:products-closed-fsb}
If $\Gamma$ has a finite star base, then
$\Inn(G)\FAut(G)$ is closed in $\Aut(G)$, and
$\Inn(G)\FInn(G)$ is closed in the relative topology of $\ISpe(G)$.
\end{proposition}

\begin{proof}
Let $A$ be a finite star base.  Suppose that a net $(\phi_i)$ in
$\Inn(G)\FAut(G)$ converges to $\phi\in\Aut(G)$, and write
$\phi_i=\ad(g_i)\tau_i$.  For every $a\in A$, choose
$1\neq x_a\in G_a$.  For all sufficiently large $i,j$, simultaneously for
all $a\in A$, one has $\phi_i(x_a)=\phi_j(x_a)$.  This non-trivial element
lies in both $g_iG_ag_i^{-1}$ and $g_jG_ag_j^{-1}$, so
Lemma~\ref{lem:vertex-conjugates} gives
$g_iG_ag_i^{-1}=g_jG_ag_j^{-1}$.
Hence
\[
 g_j^{-1}g_i\in\bigcap_{a\in A}N_G(G_a)
 =\bigcap_{a\in A}G_{N^*(a)}=G_{Z_\Gamma}.
\]

Fix one sufficiently large $j$.  By
Lemma~\ref{lem:universal-inner-factorwise}, for all sufficiently large $i$,
we have $\ad(g_j^{-1}g_i)\in\FInn(G)$.  Therefore
\begin{equation}\label{eq:coset_stabilization}
 \ad(g_i)\FInn(G)=\ad(g_j)\FInn(G),
\end{equation}
and consequently
$\ad(g_i)\FAut(G)=\ad(g_j)\FAut(G)$.  Hence,
$\phi_i\in\ad(g_j)\FAut(G)$.
This coset is closed by Proposition~\ref{prop:paut-closed}; hence
$\phi\in\ad(g_j)\FAut(G)\subseteq\Inn(G)\FAut(G)$.
This proves the first assertion.

For the second assertion, let a net in $\Inn(G)\FInn(G)$ converge to an
element of $\ISpe(G)$.  The same argument, using~\eqref{eq:coset_stabilization}, eventually places
the net in one coset of $\FInn(G)$.  By~\eqref{eq:pinn-intersection} and
Proposition~\ref{prop:paut-closed}, $\FInn(G)$ is closed in the relative
topology of $\ISpe(G)$. 
The limit therefore belongs to
$\Inn(G)\FInn(G)$.
\end{proof}

We next isolate the obstruction that will be used to prove the necessity of
star-connectedness in Theorem~\ref{thm:exact-special} below.  If $\Gamma$ is not
star-connected, then $\Gamma\setminus N^*(u)$ has at least two connected
components for some vertex $u$, and conjugating only one of these components
gives a non-trivial non-entire partial conjugation belonging to $\ISpe(G)$.
The following lemma shows that such an automorphism cannot be written as the
product of an inner automorphism and an element of $\FAut(G)$.

\begin{lemma}\label{lem:nonentire-obstruction}
A non-trivial non-entire partial conjugation is not an element of $\cl{\Inn(G)\FAut(G)}$.
\end{lemma}

\begin{proof}
Let $\pi=\pi(u,s,C)$ be a non-entire partial conjugation.  Choose
$x\in C$ and $y\in V(\Gamma)\setminus\bigl(N^*(u)\cup C\bigr)$.
Pick $1\neq g_x\in G_x$ and $1\neq g_y\in G_y$, and suppose that
$\pi$ agrees with $\ad(g)\tau$, where $g\in G$ and $\tau\in\FAut(G)$, on
the finite set $\{g_x,g_y\}$.  We show $s=1$, and therefore $\pi$ is trivial.

The automorphism $\pi$ fixes $G_y$. Since $\tau$ preserves $G_y$ and
$\pi(g_y)=\ad(g)\tau(g_y)\in G_y$,
Lemma \ref{lem:vertex-conjugates} implies $g$ normalizes $G_y$. Therefore, Fact~\ref{fact:parabolic-basics}(2) gives
$g\in G_{N^*(y)}$.
As $u\notin N^*(y)$, it follows that $u\notin\supp(g)$.
On the other hand, $\pi(g_x)=\ad(g)\tau(g_x)\in sG_xs^{-1}$,
Lemma \ref{lem:vertex-conjugates} similarly implies $g G_x g^{-1} = sG_xs^{-1}$. Therefore, $s^{-1}g\in N_G(G_x)=G_{N^*(x)}$. If $s\neq1$, then
$u\notin\supp(g)$, and the normal-form theorem gives
$u\in\supp(s)\subseteq \supp(s^{-1}g)\subseteq N^*(x)$, contradicting
$u\notin N^*(x)$.  Thus $s=1$, and hence $\pi$ is trivial.
\end{proof}

\begin{theorem}\label{thm:exact-special}
The following are equivalent.
\begin{enumerate}[(a)]
\item $\Gamma$ is star-connected.
\item $\ASpe(G)=\cl{\Inn(G)\FAut(G)}$.
\item $\ISpe(G)=\cl{\Inn(G)\FInn(G)}^{\ISpe(G)}$.
\end{enumerate}

Moreover, if $\Gamma$ has a finite star base, then the following are equivalent.
\begin{enumerate}[(1)]
\item $\Gamma$ is star-connected.
\item $\ASpe(G)=\Inn(G)\FAut(G)$.
\item $\ISpe(G)=\Inn(G)\FInn(G)$.
\end{enumerate}
\end{theorem}

\begin{proof}
If $\Gamma$ is star-connected, Theorem~\ref{thm:special-density} gives the
two density statements.  Thus (a) implies (b) and (c).

Suppose that $\Gamma$ is not star-connected.  Then for some $u$ the graph
$\Gamma\setminus N^*(u)$ has at least two connected components.  Choose a
non-trivial $s\in G_u$ and let $C$ be one, but not all, of those components.
The partial conjugation $\pi(u,s,C)$ belongs to $\ISpe(G)$, and hence to
$\ASpe(G)$, but Lemma~\ref{lem:nonentire-obstruction} shows that it does not
belong to $\cl{\Inn(G)\FAut(G)}$.  Thus (b) fails.  Since
$\cl{\Inn(G)\FInn(G)}^{\,\ISpe(G)}
\subseteq \cl{\Inn(G)\FAut(G)}$,
condition (c) fails as well.  Hence both (b) and (c) imply (a).

If $\Gamma$ has a finite star base and is star-connected, then Proposition~\ref{prop:products-closed-fsb}
turns both density statements into equality.  Thus (1) implies (2) and (3).

The reverse implications, (2)$\Rightarrow$(1) and (3)$\Rightarrow$(1), follow respectively from (b)$\Rightarrow$(a) and (c)$\Rightarrow$(a).
\end{proof}

The following proof is a slight modification of the implication
(2)$\Rightarrow$(1) in Proposition~\ref{prop:irreducible-fsb}, allowing
$\Gamma$ to be reducible.

\begin{lemma}\label{lem:inn-discrete-implies-fsb}
If $\Inn(G)$ is discrete in $\Aut(G)$, then $\Gamma$ has a finite star base.
\end{lemma}

\begin{proof}
By Lemma~\ref{lem:inn-discrete-centralizer}, choose a finite set
$F\subseteq G$ such that $C_G(F)=Z(G)$, and put
$A_0=\supp(F)$.  For every non-universal vertex
$a\in A_0\setminus Z_\Gamma$, choose $a'\notin N^*(a)$, and set
$$A=A_0\cup\{a':a\in A_0\setminus Z_\Gamma\}.$$
We show that every non-universal vertex has a non-neighbor in $A$.  This is
clear for a non-universal $v\in A_0$, using $v'$.  Suppose that
$v\notin A_0$ is non-universal, and choose $1\neq x\in G_v$.  The element
$x$ is not central in $G$: choose a non-neighbor $y$ of $v$ and a
non-trivial element of $G_y$, which does not commute with $x$ by normal
forms.  Hence $x\notin C_G(F)$, and some $f\in F$ does not commute with
$x$.  If $v$ were adjacent to every vertex in $\supp(f)$, then $x$ would
commute with every syllable of $f$, a contradiction.  Thus some vertex of
$A_0$ is a non-neighbor of $v$.  It follows that
$\bigcap_{a\in A}N^*(a)=Z_\Gamma$,
so $A$ is a finite star base.
\end{proof}

\begin{lemma}\label{lem:product-discrete}
Let $(X_i)_{i\in I}$ be nonempty topological spaces.  The product
$X=\prod_{i\in I}X_i$ is discrete if and only if every $X_i$ is discrete and
$|X_i|=1$ for all but finitely many $i$.
\end{lemma}

\begin{proof}
Suppose first that $X$ is discrete.  Every factor $X_i$ is discrete because
the coordinate projection $X\to X_i$ is an open surjection.  Fix
$x\in X$ and choose a basic open neighborhood of $x$ contained in
$\{x\}$.  Such a neighborhood restricts only finitely many coordinates.
Hence every unrestricted factor must be a singleton, so $|X_i|=1$ for all
but finitely many $i$.

Conversely, after discarding the singleton factors, $X$ is a finite product
of discrete spaces, and is therefore discrete.
\end{proof}

The group $\Inn(G)$ is normal in $\Aut(G)$, while $\FInn(G)$ and
$\FAut(G)$ are subgroups.  Therefore both products
$\Inn(G)\FInn(G)$ and $\Inn(G)\FAut(G)$
are subgroups of $\Aut(G)$.

\begin{theorem}\label{thm:pinn-discrete}
The following are equivalent.
\begin{enumerate}[(1)]
\item $\Inn(G)\FInn(G)$ is discrete in $\Aut(G)$.
\item $\Gamma$ has a finite star base, every $\Inn(G_v)$ is discrete in
$\Aut(G_v)$, and $\Inn(G_v)=1$ for all but finitely many vertices $v$.
\end{enumerate}
\end{theorem}

\begin{proof}
Assume (1).  Both $\Inn(G)$ and $\FInn(G)$ are subgroups of the discrete
group $\Inn(G)\FInn(G)$, hence are discrete.  By
Lemma~\ref{lem:inn-discrete-implies-fsb}, $\Gamma$ has a finite star base.
By Lemma~\ref{lem:factorwise-product-topology} and
Lemma~\ref{lem:product-discrete}, every $\Inn(G_v)$ is discrete and all but
finitely many of them are trivial.  This proves (2).

Conversely, assume (2).  Let $A$ be a finite star base, and choose
$1\neq x_a\in G_a$ for every $a\in A$.  Let
$W=\{v\in V(\Gamma):\Inn(G_v)\neq1\}$;
this is finite.  For every $w\in W$, discreteness of $\Inn(G_w)$ gives a
finite set $F_w\subseteq G_w$ whose pointwise stabilizer in $\Inn(G_w)$ is
trivial.  Put $$F=\{x_a:a\in A\}\cup\bigcup_{w\in W}F_w.$$
We prove that the pointwise stabilizer of $F$ in
$\Inn(G)\FInn(G)$ is trivial.

Let $\phi=\ad(g)\tau$, with $\tau\in\FInn(G)$, and suppose that $\phi$
fixes $F$ pointwise.  For every $a\in A$, the non-trivial element
$\tau(x_a)$ belongs to $G_a$, while $g\tau(x_a)g^{-1}=x_a\in G_a$.
Thus $gG_ag^{-1}\cap G_a\neq1$, and
Lemma~\ref{lem:vertex-conjugates} gives
$gG_ag^{-1}=G_a$.  Hence
\[
 g\in\bigcap_{a\in A}N_G(G_a)
   =G_{\bigcap_{a\in A}N^*(a)}=G_{Z_\Gamma}.
\]
By Lemma~\ref{lem:universal-inner-factorwise}, $\ad(g)\in\FInn(G)$.  Thus
$\phi$ itself belongs to $\FInn(G)$.  For every $w\in W$, the $w$-coordinate of $\phi$ fixes $F_w$ and is
therefore the identity.  At every vertex outside $W$, the local inner
automorphism group is trivial.  Hence $\phi=1$.  The finite set
$F$ isolates the identity, proving (1).
\end{proof}

\begin{remark}
The group $\Inn(G_v)$ is trivial exactly when $G_v$ is abelian.  It is
discrete exactly when some finite $F_v\subseteq G_v$ satisfies
$C_{G_v}(F_v)=Z(G_v)$, by
Lemma~\ref{lem:inn-discrete-centralizer}.
\end{remark}

Analogously, we prove the following.

\begin{theorem}\label{thm:paut-discrete}
The following are equivalent.
\begin{enumerate}[(1)]
\item $\Inn(G)\FAut(G)$ is discrete in $\Aut(G)$.
\item $\Gamma$ has a finite star base, every $\Aut(G_v)$ is discrete in the
pointwise-convergence topology, and $\Aut(G_v)=1$ for all but finitely many
vertices $v$.
\end{enumerate}
\end{theorem}

\begin{proof}
If (1) holds, then $\Inn(G)$ and $\FAut(G)$ are discrete subgroups.
Lemma~\ref{lem:inn-discrete-implies-fsb} gives a finite star base, while
Lemmas~\ref{lem:factorwise-product-topology} and~\ref{lem:product-discrete}
give all the local assertions.

Conversely, let $A$ be a finite star base and choose
$1\neq x_a\in G_a$ for $a\in A$.  Let
$W=\{v\in V(\Gamma):\Aut(G_v)\neq1\}$,
which is finite.  For every $w\in W$, choose a finite set
$F_w\subseteq G_w$ whose pointwise stabilizer in $\Aut(G_w)$ is trivial.
Such a set exists because $\Aut(G_w)$ is discrete.  Set
$$F=\{x_a:a\in A\}\cup\bigcup_{w\in W}F_w.$$

Suppose $\phi=\ad(g)\tau$, with $\tau\in\FAut(G)$, fixes $F$ pointwise.
Exactly as in the proof of Theorem~\ref{thm:pinn-discrete}, fixing the
non-trivial elements $x_a$ forces
$gG_ag^{-1}=G_a$ for every $a\in A$, and hence
$g\in G_{Z_\Gamma}$.  Therefore $\ad(g)\in\FInn(G)\leq\FAut(G)$, so
$\phi\in\FAut(G)$.  Its $w$-coordinate fixes $F_w$ for every $w\in W$ and
is therefore trivial; all remaining coordinate groups are already trivial.
Thus $\phi=1$, and the product group is discrete.
\end{proof}

\begin{remark}
The group $\Aut(G_v)$ is discrete precisely when a finite subset of $G_v$
determines every automorphism by its pointwise values.  In particular, this
holds whenever $G_v$ is finitely generated.  A non-trivial group has trivial
automorphism group only when it is $C_2$: triviality of inner automorphisms
forces the group to be abelian, inversion then forces exponent two, and a
vector space over $\mathbb F_2$ of dimension greater than one has non-trivial
linear automorphisms.
\end{remark}

\begin{corollary}\label{cor:special-discrete}
Assume that $\Gamma$ is star-connected and has a finite star base.  Then
$\ISpe(G)$ is discrete if and only if every $\Inn(G_v)$ is discrete and all
but finitely many are trivial.  Likewise, $\ASpe(G)$ is discrete if and only
if every $\Aut(G_v)$ is discrete and all but finitely many are trivial.
\end{corollary}

\begin{proof}
Apply Theorem~\ref{thm:exact-special} and then
Theorems~\ref{thm:pinn-discrete} and~\ref{thm:paut-discrete}.
\end{proof}

\section{Finite component witness sets}\label{sec:component-witness}

Just as finite star bases served as the key ingredient in determining discreteness and closedness of special subgroups for star-connected graphs, finite component witness sets determine the same properties for general graphs. In this section, we prove results analogous to the previous section for general graphs. To do so, we rely on the following technical lemmas.

For a subset $K\subseteq G$, put
$\supp(K)=\bigcup_{k\in K}\supp(k)$.

\begin{lemma}\label{lem:minimal-conjugator}
Suppose that $p$ has minimal syllable length in the right coset $pG_{N^*(x)}$.  Then
$\last(p)\cap N^*(x)=\varnothing$
and
\begin{equation}\label{eq:support-conjugate-factor}
 \supp(pG_xp^{-1})=\supp(p)\cup\{x\}.
\end{equation}
\end{lemma}

\begin{proof}
If a vertex $z\in\last(p)\cap N^*(x)$ existed, a reduced representative of
$p$ could be chosen with final syllable $t\in G_z$.  Since
the whole group $G_z$ lies in $G_{N^*(x)}$, the element $pt^{-1}$ would lie
in the same right coset and would have smaller syllable length, contrary to
minimality.  The support formula now follows from
Lemma~\ref{lem:reduced-conjugate}, applied to every non-trivial element of
$G_x$.
\end{proof}

The following lemma extends the right-angled Coxeter statement
\cite[Lemma~2.17]{HyttinenPaolini2019} to arbitrary graph products.  The
right-angled Coxeter version is also the support-propagation input used in
\cite[Fact~5.5]{PaoliniRabideau2026}.

\begin{lemma}\label{lem:support-propagation}
Let $\sigma\in\ASpe(G)$, let $v\in V(\Gamma)$, and let $C$ be a connected
component of $\Gamma\setminus N^*(v)$.  If $x,y\in C$ and $
 v\in \supp(\sigma(G_x))$,
then $v\in \supp(\sigma(G_y))$.
\end{lemma}

\begin{proof}
It is enough to prove the assertion when $x$ and $y$ are adjacent in
$\Gamma$, and then propagate along a path in $C$.

Choose $p_x,p_y\in G$ such that
$\sigma(G_x)=p_xG_xp_x^{-1}$ and
$\sigma(G_y)=p_yG_yp_y^{-1}$.
Choose $p_x$ and $p_y$ with minimal syllable length in the right cosets
$p_xG_{N^*(x)}$ and $p_yG_{N^*(y)}$, respectively.

Because $G_x$ and $G_y$ commute, their images under $\sigma$ commute.  Put
$q=p_x^{-1}p_y$.  After conjugating by $p_x^{-1}$, this says
$qG_yq^{-1}\leq C_G(G_x)\leq G_{N^*(x)}$.
Lemma~\ref{lem:parabolic-factorization} gives
$q=hn$ with $h\in G_{N^*(x)}$ and $n\in G_{N^*(y)}$.
Equivalently,
\begin{equation}\label{eq:conjugator-comparison}
 p_x=p_yn^{-1}h^{-1}.
\end{equation}

Since $x,y\in\Gamma\setminus N^*(v)$, we have
$v\notin N^*(x)\cup N^*(y)$.
By the hypothesis and~\eqref{eq:support-conjugate-factor},
$v\in\supp(p_x)$.  If
$v\notin\supp(\sigma(G_y)) = \supp(p_yG_yp_y^{-1})$, then, as $v\neq y$,
again by~\eqref{eq:support-conjugate-factor} we would have
$v\notin\supp(p_y)$.  Neither $h$ nor $n$ has a $v$-syllable.  Thus the
right-hand side of~\eqref{eq:conjugator-comparison} has no $v$ in its
support, by~\eqref{eq:support-product}, contradicting
$v\in\supp(p_x)$.  Thus $v\in\supp(\sigma(G_y))$.
\end{proof}

\begin{lemma}\label{lem:comp_witness_fixing}
Suppose that $\Gamma$ has a component witness set $S$ and that
$\sigma\in\ASpe(G)$ satisfies $\sigma(G_s)=G_s$ for every $s\in S$.  Then
$\sigma\in\FAut(G)$.
\end{lemma}

\begin{proof}
Suppose otherwise.  Choose $x\in V(\Gamma)$ with
$\sigma(G_x)\neq G_x$, and choose a shortest conjugator $g\in G$ such that
$\sigma(G_x)=gG_xg^{-1}$.  Then $g\neq1$; choose
$v\in\last(g)$.  By Lemma~\ref{lem:minimal-conjugator},
$v\notin N^*(x)$ and $v\in\supp(\sigma(G_x))$.

Let $C$ be the connected component of $\Gamma\setminus N^*(v)$ containing
$x$.  Since $S$ is a component witness set, choose $s\in C\cap S$.
Lemma~\ref{lem:support-propagation} gives
$v\in\supp(\sigma(G_s))$.  But $\sigma(G_s)=G_s$, so this support is
$\{s\}$.  Hence $v=s$, contradicting
$s\in\Gamma\setminus N^*(v)$.  Therefore $\sigma$ preserves every vertex
group setwise and belongs to $\FAut(G)$.
\end{proof}

With these technical lemmas established, we are now able to prove results analogous to those in the star-connected case. We note also that if $\Gamma$ has a finite component witness set $S$, then
each graph $\Gamma\setminus N^*(u)$ has only finitely many connected
components: distinct components meet $S$ in disjoint non-empty subsets, so
there are at most $|S|$ of them.  Consequently, every partial conjugation is
a finite product of single-component partial conjugations, and hence
$\PartConj(G)=\PartConjsc(G)$.

Also, the groups $\FAut(G)$ and $\FInn(G)$ normalize $\PartConj(G)$.  Indeed,
for $\alpha\in\FAut(G)$ one has
\[
 \alpha\,\pi(u,s,C)\,\alpha^{-1}=\pi(u,\alpha(s),C).
\]
Consequently, $\PartConj(G)\FAut(G)$ and
$\PartConj(G)\FInn(G)$ are subgroups of $\Aut(G)$.

\begin{proposition}\label{prop:product_closed}
If $\Gamma$ has a finite component witness set, then the following hold:
\begin{enumerate}[(1)]
\item $\PartConj(G)\FAut(G)$ is closed in $\Aut(G)$;
\item $\PartConj(G)\FInn(G)$ is closed in the relative topology of
$\ISpe(G)$.
\end{enumerate}
\end{proposition}
\begin{proof}
Let $S$ be a finite component witness set. Suppose that a net
$\phi_i=\pi_i \tau_i\in\PartConj(G)\FAut(G)$
converges to $\phi\in\Aut(G)$. For every $s\in S$, choose
$1\neq x_s\in G_s$. For all sufficiently large $i,j$, simultaneously for
all $s\in S$, one has $\phi_i(x_s)=\phi_j(x_s)$. This nontrivial element lies in both $\phi_i(G_s)$ and $\phi_j(G_s)$. Since both of these sets are conjugates of $G_s$,  Lemma~\ref{lem:vertex-conjugates} gives $\phi_i(G_s) = \phi_j(G_s)$.

Therefore, there is $i_0$ such that for all $i \geq i_0$ and all $s \in S$, we have $\phi_i(G_s) = \phi_{i_0}(G_s)$. Let $\delta_i \coloneq \phi_i^{-1} \phi_{i_0} \in \ASpe(G)$. Thus, $\delta_i(G_s) = G_s$ for all $s \in S$. Lemma \ref{lem:comp_witness_fixing} implies $\delta_i(G_v) = G_v$ for all $v \in V(\Gamma)$ and $i \geq i_0$. Thus, $\phi_i(G_v) = \phi_{i_0}(G_v)$ for all $v \in V(\Gamma)$. Hence
$\phi_i^{-1}\phi_{i_0}\in\FAut(G)$,
so the tail of the net $(\phi_i)$ lies in the closed coset
$\phi_{i_0}\FAut(G)$. Therefore
$\phi\in\phi_{i_0}\FAut(G)\subseteq\PartConj(G)\FAut(G)$,
proving the first assertion.

For the second assertion, if moreover
$\phi_i\in\PartConj(G)\FInn(G)$ and $\phi_i\to\phi\in\ISpe(G)$, then
$\phi_i^{-1}\phi_{i_0}\in\FAut(G)\cap\ISpe(G)=\FInn(G)$. Hence the tail
lies in the relatively closed coset $\phi_{i_0}\FInn(G)$, and therefore
$\phi\in\phi_{i_0}\FInn(G)\subseteq\PartConj(G)\FInn(G)$.
\end{proof}

\begin{corollary}\label{cor:product_equality}
    Let $G = \Gamma\cG$. If $\Gamma$ has a finite component witness set, then \[
 \ASpe(G)=\FAut(G)\PartConjsc(G)
\]
and
\[
 \ISpe(G)=
 \FInn(G)\PartConjsc(G).
\]
\end{corollary}
\begin{proof}
    For any graph $\Gamma$, we have
\[
 \FAutfin(G)\PartConjsc(G) \leq \FAut(G)\PartConjsc(G) \leq \ASpe(G).
\]
When $\Gamma$ has a finite component witness set, taking closures and applying Proposition \ref{prop:product_closed}, Theorem \ref{thm:special-closedness}, Theorem \ref{thm:partial-density}, and $\PartConjsc(G) = \PartConj(G)$  yields $\FAut(G)\PartConjsc(G) = \ASpe(G)$.

    Similarly, we have
\[
 \FInnfin(G)\PartConjsc(G) \leq \FInn(G)\PartConjsc(G) \leq \ISpe(G).
\]
When $\Gamma$ has a finite component witness set, taking relative closures and applying Proposition \ref{prop:product_closed}, Theorem \ref{thm:special-closedness}, Theorem \ref{thm:partial-density}, and $\PartConjsc(G) = \PartConj(G)$  yields $\FInn(G)\PartConjsc(G) = \ISpe(G)$.
\end{proof}
\begin{lemma}\label{lem:inn-discrete-implies-fcws}
  If $\PartConj(G)$ is discrete in $\Aut(G)$, then $\Gamma$ has a finite component witness set.
\end{lemma}
\begin{proof}
Since $\PartConj(G)$ is discrete, there is a finite set $F\subseteq G$
whose pointwise stabilizer in $\PartConj(G)$ is trivial. Set
$S=\supp(F)$,
which is finite. We claim that $S$ is a component witness set.

Otherwise, there are $u\in V(\Gamma)$ and a connected component $C$ of
$\Gamma\setminus N^*(u)$ such that $C\cap S=\varnothing$. Choose
$1\neq a\in G_u$ and consider the partial conjugation
$\pi=\pi(u,a,C)$.
This partial conjugation is non-trivial: if $v\in C$ and
$1\neq x\in G_v$, then $u$ and $v$ are non-adjacent, and hence
$\pi(x)=axa^{-1}\neq x$.
On the other hand, since $C\cap\supp(F)=\varnothing$, the automorphism
$\pi$ fixes every syllable occurring in every element of $F$, and
therefore fixes $F$ pointwise. This contradicts the choice of $F$.
Thus $S$ is a finite component witness set.
\end{proof}

\begin{theorem}\label{thm:pinn-discrete-general}
The following are equivalent.
\begin{enumerate}[(1)]
\item $\PartConj(G)\FInn(G)$ is discrete in $\Aut(G)$.
\item $\Gamma$ has a finite component witness set, every $\Inn(G_v)$ is discrete in
$\Aut(G_v)$, and $\Inn(G_v)=1$ for all but finitely many vertices $v$.
\end{enumerate}
\end{theorem}

\begin{proof}
Assume (1).  Both $\PartConj(G)$ and $\FInn(G)$ are subgroups of the discrete
group $\PartConj(G)\FInn(G)$, hence are discrete.  By
Lemma~\ref{lem:inn-discrete-implies-fcws}, $\Gamma$ has a finite component witness set.
By Lemma~\ref{lem:factorwise-product-topology} and
Lemma~\ref{lem:product-discrete}, every $\Inn(G_v)$ is discrete and all but
finitely many of them are trivial.  This proves (2).

Conversely, assume (2).  Let $S$ be a finite component witness set, and choose
$1\neq x_s\in G_s$ for every $s\in S$.  Let
$W=\{v\in V(\Gamma):\Inn(G_v)\neq1\}$;
this is finite.  For every $w\in W$, discreteness of $\Inn(G_w)$ gives a
finite set $F_w\subseteq G_w$ whose pointwise stabilizer in $\Inn(G_w)$ is
trivial.  Put $$F=\{x_s:s\in S\}\cup\bigcup_{w\in W}F_w.$$
We prove that the pointwise stabilizer of $F$ in
$\PartConj(G)\FInn(G)$ is trivial.

Let $\phi=\alpha\tau$, with $\alpha \in \PartConj(G)$ and $\tau\in\FInn(G)$, and suppose that $\phi$
fixes $F$ pointwise.  For every $s\in S$, the non-trivial element
$\tau(x_s)$ belongs to $G_s$, while $\alpha(\tau(x_s))=x_s\in G_s$.
Thus $\alpha(G_s)\cap G_s\neq1$.  Since $\alpha(G_s)$ is a conjugate
of $G_s$, Lemma~\ref{lem:vertex-conjugates} gives
$\alpha(G_s)=G_s$.  Lemma~\ref{lem:comp_witness_fixing} now implies
$\alpha\in\FAut(G)$.  Since also
$\alpha\in\PartConj(G)\leq\ISpe(G)$,
\eqref{eq:pinn-intersection} gives $\alpha\in\FInn(G)$.  Hence
$\phi\in\FInn(G)$.

For every $w\in W$, the $w$-coordinate of $\phi$ fixes $F_w$ and is
therefore the identity.  At every vertex outside $W$, the local inner
automorphism group is trivial.  Hence $\phi=1$.  This proves (1).
\end{proof}

Analogously, we prove the following.

\begin{theorem}\label{thm:paut-discrete-general}
The following are equivalent.
\begin{enumerate}[(1)]
\item $\PartConj(G)\FAut(G)$ is discrete in $\Aut(G)$.
\item $\Gamma$ has a finite component witness set, every $\Aut(G_v)$ is discrete in the
pointwise-convergence topology, and $\Aut(G_v)=1$ for all but finitely many
vertices $v$.
\end{enumerate}
\end{theorem}

\begin{proof}
Assume (1). Both $\PartConj(G)$ and $\FAut(G)$ are subgroups of the discrete
group $\PartConj(G)\FAut(G)$, hence are discrete.  By
Lemma~\ref{lem:inn-discrete-implies-fcws}, $\Gamma$ has a finite component witness set.
By Lemma~\ref{lem:factorwise-product-topology} and
Lemma~\ref{lem:product-discrete}, every $\Aut(G_v)$ is discrete and all but
finitely many of them are trivial.  This proves (2).

Conversely, assume (2).  Let $S$ be a finite component witness set, and
choose $1\neq x_s\in G_s$ for every $s\in S$.  Let
$W=\{v\in V(\Gamma):\Aut(G_v)\neq1\}$;
this is finite.  For every $w\in W$, discreteness of $\Aut(G_w)$ gives a
finite set $F_w\subseteq G_w$ whose pointwise stabilizer in $\Aut(G_w)$ is
trivial.  Put $$F=\{x_s:s\in S\}\cup\bigcup_{w\in W}F_w.$$
We prove that the pointwise stabilizer of $F$ in
$\PartConj(G)\FAut(G)$ is trivial.

Let $\phi=\alpha\tau$, with $\alpha\in\PartConj(G)$ and
$\tau\in\FAut(G)$, and suppose that $\phi$ fixes $F$ pointwise.  For every
$s\in S$, the non-trivial element $\tau(x_s)$ belongs to $G_s$, while
$\alpha(\tau(x_s))=x_s\in G_s$.
Hence $\alpha(G_s)\cap G_s\neq1$, and
Lemma~\ref{lem:vertex-conjugates} gives $\alpha(G_s)=G_s$.
Lemma~\ref{lem:comp_witness_fixing} now implies
$\alpha\in\FAut(G)$.  Since
$\alpha\in\PartConj(G)\leq\ISpe(G)$,
\eqref{eq:pinn-intersection} gives
$\alpha\in\FInn(G)\leq\FAut(G)$.  Therefore $\phi\in\FAut(G)$.
Its $w$-coordinate fixes $F_w$ for every $w\in W$ and is therefore the
identity; all remaining coordinate groups are already trivial.  Thus
$\phi=1$, and the product group is discrete.
\end{proof}

\section{Special automorphisms of graph products of abelian groups}
\label{sec:special-abelian}

Throughout this section, every vertex group is a non-trivial abelian group.
Since every $G_v$ is abelian, $\Inn(G_v)=1$, and therefore
\begin{equation}\label{eq:abelian-finn-trivial}
 \FInn(G)=1.
\end{equation}

\subsection{A topological semidirect decomposition}

\begin{theorem}\label{thm:abelian-splitting}
There is a canonical continuous homomorphism
\[
 \Theta:\ASpe(G)\longrightarrow\FAut(G)
\]
which restricts to the identity on $\FAut(G)$ and has kernel $\ISpe(G)$.
Consequently, multiplication induces an isomorphism of topological groups
\[
 \ISpe(G)\rtimes\FAut(G)\xrightarrow{\ \cong\ }\ASpe(G).
\]
In particular,
\[
 \ASpe(G)/\ISpe(G)\cong
 \FAut(G)\cong\prod_{v\in V(\Gamma)}\Aut(G_v)
\]
as topological groups.
\end{theorem}

\begin{proof}
Let $\sigma\in\ASpe(G)$.  For each $v$, choose $p_v\in G$ such that
$\sigma(G_v)=p_vG_vp_v^{-1}$, and define
$\Theta_v(\sigma)=\ad(p_v^{-1})\circ\sigma|_{G_v}\in\Aut(G_v)$.
This does not depend on the choice of $p_v$.  Indeed, if $p_v'$ is another
choice, then $(p_v')^{-1}p_v\in N_G(G_v)=G_{N^*(v)}$.  Conjugation by an
element of $G_{N^*(v)}$ induces on $G_v$ the inner automorphism determined by
its $G_v$-coordinate.  Since $G_v$ is abelian, this induced automorphism is
the identity.

Let $\sigma,\tau\in\ASpe(G)$, and choose conjugators $p_v,q_v$ for them,
respectively. 

For $x\in G_v$ we have $\tau(x)=q_v\Theta_v(\tau)(x)q_v^{-1}$
and hence
\[
 (\sigma\tau)(x)=
 \sigma(q_v)p_v
 \Theta_v(\sigma)\bigl(\Theta_v(\tau)(x)\bigr)
 p_v^{-1}\sigma(q_v)^{-1}.
\]
Thus $\sigma(q_v)p_v$ is a conjugator for $(\sigma\tau)(G_v)$, and the
well-definedness just proved gives
$\Theta_v(\sigma\tau)=\Theta_v(\sigma)\Theta_v(\tau)$.
The family $(\Theta_v(\sigma))_v$ therefore defines a homomorphism
$\Theta:\ASpe(G)\longrightarrow\prod_v\Aut(G_v)=\FAut(G)$.
If $\sigma\in\FAut(G)$, one can take every $p_v=1$, so
$\Theta(\sigma)=\sigma$.  Moreover, $\Theta(\sigma)=1$ exactly when
$\sigma|_{G_v}=\ad(p_v)|_{G_v}$ for every $v$, which is precisely the
condition $\sigma\in\ISpe(G)$.  This proves exactness and the algebraic
semidirect decomposition.

It remains to check the topology.  We show that every coordinate
$\Theta_v$ is continuous.  Let a net $(\sigma_i)$ in $\ASpe(G)$ converge to
$\sigma$, and choose $1\neq x\in G_v$.  Eventually
$\sigma_i(x)=\sigma(x)$.  This common non-trivial element belongs to both
$\sigma_i(G_v)$ and $\sigma(G_v)$, so
Lemma~\ref{lem:vertex-conjugates} gives $\sigma_i(G_v)=\sigma(G_v)$
for all sufficiently large $i$.  Fix a conjugator $p_v$ for
$\sigma(G_v)$.  For those indices it can also be used in the definition of
$\Theta_v(\sigma_i)$, and for every $y\in G_v$ we have eventually
$\Theta_v(\sigma_i)(y)=p_v^{-1}\sigma_i(y)p_v
=p_v^{-1}\sigma(y)p_v=\Theta_v(\sigma)(y)$.
Thus $\Theta_v$ is continuous in the pointwise-convergence topology, and
therefore so is $\Theta$ into the product.  By
Lemma~\ref{lem:factorwise-product-topology}, the canonical section
$\FAut(G)\hookrightarrow\ASpe(G)$ is a topological embedding.  Hence
\[
 \sigma\longmapsto
 \bigl(\sigma\Theta(\sigma)^{-1},\Theta(\sigma)\bigr)
\]
is a continuous inverse to multiplication
$\ISpe(G)\rtimes\FAut(G)\to\ASpe(G)$.
\end{proof}

\subsection{Right-angled Coxeter groups}

We now formulate our previous results in the setting of right-angled Coxeter groups. Let $W=W(\Gamma)$ be the graph product with $G_v=C_2$ for every $v$.  We
write $\Spe(W)$ for Tits' special automorphism group: the group of
$\alpha\in\Aut(W)$ such that every involution of $W$ is sent to a conjugate
of itself.

\begin{proposition}\label{prop:racg-identification}
For every right-angled Coxeter group,
\[
 \FAut(W)=\FInn(W)=1,
 \qquad
 \Spe(W)=\ASpe(W)=\ISpe(W).
\]
\end{proposition}

\begin{proof}
Since $\Aut(C_2)=1$, we have $\FAut(W)=\FInn(W)=1$ and
$\ASpe(W)=\ISpe(W)$, while $\Spe(W)\leq\ASpe(W)$ is immediate.  Conversely,
in Tits' decomposition $\Aut(W)=\Spe(W)\rtimes F(\Gamma)$
\cite{Tits1988}, the $F(\Gamma)$-coordinate is determined by the images of
the standard generators in $W_{\mathrm{ab}}$; an element of $\ASpe(W)$
sends every generator to a conjugate of itself and therefore has trivial
$F(\Gamma)$-coordinate.
\end{proof}

Although assertion~(1) and the ``if'' direction of~(2) are new, all the remaining assertions can be extracted from \cite{PaoliniRabideau2026}.  For completeness, we state the full set of results on right-angled Coxeter groups here.

\begin{theorem}\label{thm:racg-density-equality}
Let $W=W(\Gamma)$ be a right-angled Coxeter group.
\begin{enumerate}[(1)]
\item For every defining graph, we have $
 \Spe(W)=\cl{\PartConj(W)}$.
\item The graph $\Gamma$ is star-connected if and only if $
 \Spe(W)=\cl{\Inn(W)}$.
\item $\Spe(W)=\Inn(W)$ if and only if $\Gamma$ is star-connected
and $\Inn(W)$ is closed in $\Aut(W)$.
\item If $\Gamma$ has a finite star base, then $\Spe(W)=\Inn(W)$ if
and only if $\Gamma$ is star-connected.
\item If $W$ is countable, then $\Spe(W)=\Inn(W)$ if and only if $\Gamma$ is
star-connected and has a finite star base.
\end{enumerate}
\end{theorem}

\begin{proof}
The first assertion is Theorem \ref{thm:partial-density} together with Proposition~\ref{prop:racg-identification}.

The second and fourth assertions are Theorem \ref{thm:exact-special} and Proposition~\ref{prop:racg-identification}.

The third assertion is Theorem \ref{thm:exact-special}, Theorem \ref{thm:special-closedness}, and Proposition~\ref{prop:racg-identification}.

Finally, if $W$ is countable, Theorem \ref{thm:exact-special} gives the 'if' direction. For the 'only if' direction, since $W$ is countable, $\Aut(W)$ is Polish. Thus, the Baire category theorem implies countable closed subgroups of $\Aut(W)$ are discrete. Hence, suppose that $\Spe(W)=\Inn(W)$. Since $W$ is countable, so is $\Inn(W)$; it is closed by Theorem~\ref{thm:special-closedness}. It is therefore discrete. Thus Lemma \ref{lem:inn-discrete-implies-fsb} implies $\Gamma$ has a finite star base, and then Theorem \ref{thm:exact-special} implies $\Gamma$ is star-connected.
\end{proof}

\end{document}